\UseRawInputEncoding
\documentclass[10pt]{article}

\usepackage{amsfonts}
\usepackage{amsmath,amsthm,amscd,amssymb,mathrsfs,setspace}
\usepackage{latexsym,epsf,epsfig}
\usepackage{cite}
\usepackage{color}
\usepackage[hmargin=1in,vmargin=1in]{geometry}
\usepackage{hyperref}

\newcommand{\ds}{\displaystyle}

\newcommand{\cA}{{\mathcal{A}}}

\newcommand{\cE}{{\mathcal{E}}}

\newcommand{\N}{\mathbb{N}}
\newcommand{\cB}{\mathcal{B}}
\newcommand{\bu}{\mathbf u}
\newcommand{\bv}{\mathbf v}

\newcommand{\bF}{\mathbf F}

\newcommand{\bV}{\mathbf V}

\newcommand{\grad}{\nabla}

\theoremstyle{plain}
\newtheorem{theorem}{Theorem}[section]
\newtheorem{lemma}[theorem]{Lemma}

\theoremstyle{remark}
\newtheorem{definition}{Definition}
\newtheorem{remark}{Remark}[section]
\numberwithin{equation}{section}
\numberwithin{theorem}{section}
\numberwithin{remark}{section}
\numberwithin{assumption}{section}
\numberwithin{condition}{section}
\title{Mathematical Effects of Linear Visco-elasticity \\[.1cm]  in Quasi-static Biot Models}
\author{ Lorena Bociu\footnote{2311 Stinson Dr., North Carolina State University, Raleigh, NC, 27607; {\em lvbociu@ncstate.edu}} \hskip 2cm Boris Muha\footnote{University of Zagreb, Faculty of Science, Croatia;~ {\em borism@math.hr}}  \hskip 2cm Justin T. Webster\footnote{1000 Hilltop Dr., University of Maryland, Baltimore County, Baltimore, MD, 21250;~ {\em websterj@umbc.edu}} }

\begin{document}

\maketitle
\begin{abstract}
\noindent We investigate and clarify the mathematical properties of linear poro-elastic systems {in the presence of classical (linear, Kelvin-Voigt) visco-elasticity}. In particular, we quantify the time-regularizing and dissipative effects of visco-elasticity in the context of the quasi-static Biot equations. The full, coupled pressure-displacement presentation of the system is utilized, as well as the  framework of implicit, degenerate evolution equations, to demonstrate such effects and characterize linear {\em poro-visco-elastic systems}. We consider a simple presentation of the  dynamics (with convenient boundary conditions, etc.) for clarity in exposition across several relevant parameter ranges. Clear well-posedness results are provided, with associated a priori estimates on the solutions. In addition, precise statements of admissible initial conditions in each scenario are given. 
\noindent 
\vskip.2cm

\noindent Keywords: {poroelasticity, implicit evolution equations, strong damping, viscoelasticity}
\vskip.2cm
\noindent
{\em 2010 AMS}: 74F10, 76S05, 35M13, 35A01, 35B65, 35Q86, 35Q92
\vskip.2cm
\noindent Acknowledged Support: Author 1 was supported by NSF-DMS 1555062 (CAREER). Author 2 was supported by the Croatian Science Foundation project IP-2019-04-1140. Author 3 was partially supported by NSF-DMS 1907620. 
\end{abstract}

\section{Introduction} \noindent
In the past 10 years, there has been an intense growth of work in theoretical and numerical studies invoking  equations of poro-elasticity \cite{both2,banks2, bmw,cao2,showrecent, numerical1,mikelic,MRT,rohan,applied2,applied3,numerical2} (to name just a few).  While the initial development of the mathematical theory of poro-elasticity was driven by geophysical applications \cite{Terzaghi, Biot, detournay, show1,applied, wheel}, some of the recent interest in this field seems due to the fact that deformable porous media models describe biological tissues; these include organs, cartilages, bones and engineered tissue scaffolds \cite{bcmw, bgsw, MBE2,bociuno,nia,ozkaia,GGbook,biobiot}.
The mechanics of biological tissues typically exhibit {\em both elastic and visco-elastic behaviors}, resulting from the combined action of elastin and collagen \cite{mow, nia, ozkaia}. These effects can change over time, and the loss of tissue visco-elasticity is relevant to the study of several age-related diseases such as: glaucoma, atherosclerosis, and Alzheimer's disease \cite{GGbook}. {Several mathematically-oriented} studies have invoked or utilized {visco-elastic effects} in the dynamics, owing to their analytically and numerically regularizing and dissipative properties, e.g.,  \cite{bgsw,both4,rohan}. Thus, there is both mathematically-driven and application-driven motivation to consider a comprehensive mathematical investigation of poro-visco-elastic systems. Poro-visco-elastic solids were considered by Biot himself in \cite{Biotporovisco}; the seminal poro-elasticity reference \cite{coussy} contains a discussion of the modeling of poro-visco-elastic solids, and we also point to the references \cite{EveVernescu91,Sanchzez-Palencia} in this regard.

The field of visco-elastic solids is vast, but we note that for purely hyperbolic-like dynamics (bulk, plate/shell, or beam elasticity), the effects of classical, linear visco-elasticity are well-understood and fully characterized at the abstract level \cite{redbook, trig2, trig3}. And, for classical Biot-type systems of porous-elastic dynamics, the abstract theory of solutions is established, with clear estimates on solutions (and we recap this in detail below). Yet, {\em for linear poro-visco-elastic models, there seems to be no comprehensive presentation in the literature demarcating which parabolic behaviors are present, with clear estimates on solutions quantifying solution regularity and dissipation}. There have been recent numerical investigations into the effects of visco-elasticity in Biot type models involving the first author here \cite{MBE1,MBE2}, as well as studies where poro-visco-elastic systems are studied numerically \cite{rohan,gaspar,both3}. In this article, we investigate the general  linear quasi-static poro-visco-elastic system, and clarify the time-regularizing and dissipative effects of structural visco-elasticity; in particular, we map out the theory of partial differential equation (PDE) solutions for linear poro-visco-elastic systems.

In this work we focus on perhaps the simplest inclusion of visco-elasticity: Kelvin-Voigt type. This is precisely what was considered in \cite{bgsw}. Here, we  build on that mathematical framework, where weak solutions were constructed in a particular 3D case, and the subsequent 1D investigations of viscous effects \cite{MBE1,MBE2}.  In the present context, we provide clear well-posedness results with estimates, and a discussion of the construction of solutions (when illustrative). In addition, we make precise an appropriate notion of initial conditions in each relevant scenario. When it is appropriate, we relate the system abstractly to an associated semigroup framework \cite{pazy,redbook}. We do not attempt to describe or invoke more sophisticated or recent theories of visco-elasticity here. Indeed, there does not yet seem to be a comprehensive PDE theory of poro-visco-elasticity in the simple, linear case of Kelvin-Voigt structural viscosity.

The popular quasi-static formulation (neglecting elastic inertia) of Biot's equations is utilized here. On the other hand, the ``full" inertial Biot system is formally equivalent \cite{show1} to thermo-elasticity, which is well-studied \cite{redbook,oldthermo}. Foundational works for the PDE theory of linear poro-elasticity can be found in \cite{auriault,indiana, zenisek}, and culminating in the more modern works \cite{show1,showmono}. In this traditional framework, visco-elastic effects were considered in the displacement dynamics by invoking the so called {\em secondary consolidation} \cite{show1,applied, gaspar}, typical for studies of {\em clays}. More recently, {as described above,} a growing interest in biologically-based Biot systems can also be observed. In these bio-Biot models, linear visco-elastic effects can be incorporated into traditional linear Biot dynamics by taking into account the visco-elastic strain, and possibly adjusting the formula for the {\em local fluid content} accordingly (depending on the specific scenario considered, either focusing on incompressible or compressible constituents). We  address several parameter regimes of physical interest here. We do this in the spirit of the analysis of the well-known reference \cite{show1}; we also include some comments on the PDE effects of secondary consolidation (considered as a {\em partial visco-elasticity}) at the end of the work in Section \ref{secondary}. Although we focus on linear models with constant coefficients, recent applications---which inspired the consideration of models here---are in fact nonlinear or taken with time-dependent coefficients \cite{bw,bmw,bcmw,bgsw,mikelic}. The work here can be seen as a foundation for extending such considerations to the visco-elastic case; additionally, this work provides a clear and precise framework for researchers utilizing visco-elastic terms as model-regularizations, as is common in fluid-structure interactions, e.g. \cite{IKLT14,MT21,BorSun}. 

A main focus of this work is to introduce the appropriate constituent operators and spaces into the context visco-elastic dynamics, which have been used in abstractly describing the quasi-static Biot dynamics for some time \cite{bgsw,bw,cao2,show1,auriault,indiana}. Following this, we can ``reduce" or frame the poro-visco-elastic dynamics in the context of these operators to apply, when possible; to the knowledge of the authors, this has not been done. Interestingly, in some cases below, the abstract presentation of these systems reveals central features of poro-visco-elastic dynamics which are not immediately obvious in the full presentation of solutions. This is particularly true in considering which type of initial conditions are warranted for each configuration of interest, and connects the analysis thereof to appropriate ODE or semigroup frameworks.

\section{Quasi-static Poro-Visco-elastic Dynamics}\noindent
Let $\Omega \subset \mathbb R^n$ for $n=2,3$ be a bounded, smooth domain. For the dynamics, we follow a bulk of the notational conventions of \cite{show1}.
In the traditional fully-saturated Biot system we have a pressure equation and a momentum equation; these are given in the  variable $\bu$, describing the displacements of the solid matrix, and the homogenized {pore} pressure $p$. The pressure equation, resulting from mass balance, reads:
\begin{equation}\zeta_t-\nabla \cdot [K\nabla p]=S.\end{equation}
The quantity $\zeta$ is the {\em fluid-content}, and in the standard  Biot model of poro-elasticity it is given by 
\begin{equation} \zeta=c_0p+\alpha \nabla \cdot \bu.\end{equation}
The constant $c_0 \ge 0$ represents {\em compressibility of the constituents}, and will be considered in two regimes here: $c_0=0$ (incompressible constituents) and $c_0>0$ (compressible constituents).\footnote{We recall here that incompressibility of each component means that the volumetric deformation of the solid constituent corresponds to the variation of 
fluid volume per unit volume of
porous material, i.e., $\zeta =  \nabla \cdot \bu$.} The coupling constant $\alpha>0$ is referred to as the {\em Biot-Willis} constant, and, in the case of incompressible constituents $c_0=0$, we have $\alpha = 1$ \cite{detournay}. The quantity $K(\mathbf x,t)$ is the {\em permeability tensor} of the porous-elastic structure.  We present it generally here (as in \cite{bmw}), but {for the analysis below, we will take $k=const.$} This will provide clarity as we demonstrate the mathematical structures of poro-visco-elastic systems, and is in-line with the central mathematical references for the PDE analysis of Biot's equations, \cite{show1,auriault,indiana}. The fluid source function $S$ is permitted to depend on $\mathbf x$ and $t$.

The momentum equation {for the fluid-solid mixture} is given as an elliptic $(\mu,\lambda)$-Lam\'e system, as driven by the pressure gradient and a source $\bF$:
\begin{equation}-\mu\Delta \bu-(\lambda+\mu)\nabla \nabla \cdot \bu+\alpha\nabla p = \bF.\end{equation}
Below, we will consider the body force $\bF$ to be spatially and temporally dependent. 

For reference, we  recall that the primal, inertial form of the elasticity equation is \begin{equation} \label{inertial} \rho\bu_{tt}-\mu\Delta \bu-(\lambda+\mu)\nabla \nabla \cdot \bu+\nabla p = \bF.\end{equation} It is instructive to remember that the Biot dynamics begin here, and then take $\rho\bu_{tt} \approx 0$ to obtain the standard {\em quasi-static} equations of poro-elasticity \cite{show1,coussy}. 
 
 \subsection{Inclusion of Visco-elasticity: $\delta_1>0$}\noindent
In the most straight-forward manner, the incorporation of {visco-elasticity} may be achieved through the momentum equation.  We shall refer to this here as {\em full, linear visco-elasticity}, and follow the Kelvin-Voigt approach of including  strong (or ``structural") damping \cite{redbook,trig2,chenruss,trig3}. This entails adding a {\em strain rate} term to the global stress, including a ``strength" coefficient, $\delta_1 \ge 0$, which captures the viscous structural effects. Denoting the symmetric gradient (linearized strain) as $\varepsilon$ and standard linear elastic stress as $\sigma$ we have \begin{equation} \varepsilon(\bv) = \dfrac{1}{2}[\nabla \bv+\nabla \bv^T],~~~~\sigma(\bv) = 2\mu\varepsilon(\bv)+\lambda(\nabla\cdot \bv) I.\end{equation} This yields  the structural term 
\begin{equation}-\text{div}~\sigma(\bu+\delta_1\bu_t)=-[\mu\Delta+(\lambda+\mu)\nabla \text{div}](\bu + \delta_1\bu_t)\end{equation} in the system, where $\delta_1>0$ indicates the presence of visco-elasticity (i.e., $\delta_1=0$ represents the standard, elastic Biot dynamics). We note that the inclusion of this term is referred to as visco-elasticity, since the full inertial Biot dynamics in \eqref{inertial} would constitute a linear, visco-elastic system of elasticity (a {\em strongly damped} structural equation) with the inclusion of $\delta_1>0$. This straight-forward inclusion of viscous effects in the solid can be obtained as a limiting case of the poro-visco-elastic modelling in \cite{coussy}.
\begin{remark}
There are many ways to incorporate visco-elastic effects into the modeling of poro-elastic systems. See, for example, \cite{both3}, where a viscoelastic strain is considered in the case of compressible constituents; the other components of the system are also updated there, including the formula for the fluid content. In general, the field of visco-elasticity is broad, and we do not claim to be exhaustive. Here, our focus on the specific case of linear, quasi-static Biot in the presence of linear, Kelvin-Voigt structural viscosity. Other pertinent references for poro-visco-elastic systems are: \cite{Sanchzez-Palencia,EveVernescu91,MeiVernescu10, Meirm, Vist}. The aforementioned references include aspects of the homogenization theory, detailing when and how visco-elasticity can arise in solids, and, in some cases existence results for weak solutions.
\end{remark}
A central point here is that we permit the presence of visco-elasticity to affect the definition of the fluid content $\zeta$. In the established  reference \cite{coussy}, compressible constituents and viscous effects are considered, with modified  fluid content. There are two choices for $\zeta$ considered here, encapsulated by $\delta_2 = 0$ or $\delta_2>0$:
\begin{equation} \zeta = c_0p+ \alpha \nabla \cdot \bu + \delta_2\nabla \cdot \bu_t.\end{equation}
When $\delta_2=0$, this represents the standard Biot definition of the fluid content, which prevalently appears in the literature for linearized poro-elastic systems with and without visco-elastic effects, e.g., \cite{show1,bgsw,bsAWM}. {A derivation of the model, obtained by heterogeneous mixture approach, can be found in \cite[Section 12.2]{GGbook}, for instance}. 
In the present note, we take the approach of classifying the system and its solutions in two regimes: $\delta_2=0$ and $\delta_2>0$, noting that {the application of interest can inform which selection is made.}

To conclude this section, we re-state the linear poro-visco-elastic system as it is studied herein. We fix $\alpha, \lambda,\mu>0$, and consider the following with $\delta_1>0$:
\begin{equation}\label{PVEsystem}
\begin{cases}
-\text{div}~\sigma(\bu+\delta_1\bu_t) +\alpha\nabla p = \bF \\
[c_0p+ \alpha \nabla \cdot \bu + \delta_2\nabla \cdot \bu_t]_t  -\nabla \cdot [k\nabla p]=S, 
\end{cases}
\end{equation}
where we accommodate all possible regimes dependent on  $c_0  \geq 0$ and $\delta_2 \ge 0$. 
To \eqref{PVEsystem} we associate the following boundary conditions: 
\begin{equation}\label{PVEsystemBC}
\bu = 0, \ \text{and} \ \ k\nabla p \cdot \mathbf n  = 0 \ \ \text{on}\ \  \Gamma \equiv \partial \Omega,
\end{equation}
namely, homogeneous Dirichlet conditions for the displacement and homogeneous Neumann conditions for the pressure.

At this juncture, we suggest that the natural initial conditions to be specified are those quantities appearing under the time derivatives above, namely: $$\delta_1\bu(0)=\bu_0~~~\text{ and } ~~~[c_0p+\alpha \nabla \cdot \bu+\delta_2\nabla \cdot \bu_t](0)=d_0.$$ It is immediately clear that the regime of interest and the parameter values affect the relative independence of these quantities. In what follows, we will precisely specify the initial quantities, their relationships, and their spatial regularities, as each depend on the regime under consideration and the type of solution sought.

 Lastly, we mention a model of {\em partial visco-elasticity}, known as {\em secondary consolidation} (for certain soils, such as clays \cite{show1,applied,gaspar}), which has appeared in the literature. In this case, for $\lambda^*>0$, the displacement equation reads:
\begin{equation}
-(\mu+\lambda) \nabla \nabla \cdot \bu-\mu\Delta \bu-\lambda^*\nabla \nabla \cdot \bu_t+\alpha\nabla p = \bF
\end{equation}
We remark on secondary consolidation briefly in Section \ref{secondary} at the end of this note.

\subsection{Notation and Conventions} 
\noindent
\textbf{For the remainder of the paper, we consider $\alpha, \mu,\lambda>0$ to be fixed} and do not explicitly name them in theorems and subsequent discussions.

 The Sobolev space of order $s$ defined on a domain $D$ will be denoted by $ H^s(D)$, with $H^s_0(D)$ denoting the closure of test functions $C_0^{\infty}(D) := \mathcal D(D)$ in the standard $H^s(D)$ norm
(which we denote by $\|\cdot\|_{H^s(D)}$ or $\|\cdot\|_{s,D}$).  When $s =0 $ we may further abbreviate the notation
to $\| \cdot \|$ denoting $\|\cdot\|_{L^2(D)}$. Vector-valued spaces will be denoted as $\mathbf L^2(\Omega) \equiv [L^2(\Omega)]^n$ and $\mathbf H^s(\Omega) = [H^s(\Omega)]^n$. We make use of the standard notation for the trace of a function $w$ as $\gamma[w]$, namely $\gamma[\cdot]$ as the map from $H^1(D)$ to $H^{1/2}(\partial D)$. We will make use of the spaces $L^2(0,T;U)$ and $H^s(0,T;U)$, when $U$ is a Hilbert space. Associated norms (and inner products) will be denoted with the appropriate subscript, e.g., $||\cdot||_{L^2(0,T;U)}$, though we will simply denote $L^2$-inner products by $(\cdot,\cdot)$ when the context is clear. For estimates on solutions, we utilize the notation that $A \lesssim B$ means there exists a constant $c>0$ not depending on critical constants (made clear by context) so that $A \le cB$.

\subsection{Operators, Spaces, and Solutions}\label{opsec}
Let   $\bV=\mathbf H_0^1(\Omega)$ and $V=H^1(\Omega)\cap L^2_0(\Omega)$.\footnote{$L^2_0(\Omega) \equiv \{ f \in L^2(\Omega)~:~\int_{\Omega} f = 0\}$ which is  topologically isomorphic to $L^2(\Omega)/\mathbb R$.} We will topologize $V$ through the inner-product
$$a(p,q) = (k\nabla p,\nabla q)_{L^2(\Omega)},$$ which gives rise to the
 gradient norm on $V$; by  the Poincar\'e-Wirtinger inequality,  the norm
$$||\cdot||_V = ||k^{1/2}\nabla \cdot||,$$
is equivalent to the standard $H^1(\Omega)$. Through Korn's inequality, as well as Poincar\'e's inequality, we may topologize $\bV$ through the bilinear form:
$$e(\bu,\bv) = (\sigma(\bu),\varepsilon(\bv))_{\mathbf L^2(\Omega)},$$
with $\sigma,~\varepsilon$ defined as above, leading to the $\mathbf H^1(\Omega)$-equivalent norm
$$||\cdot||_{\mathbf V} = e(\cdot,\cdot)^{1/2}.$$

We define two linear differential operator associated to the bilinear forms $e(\cdot,\cdot)$ and $a(\cdot,\cdot)$, with (respectively) actions given by \begin{equation}\label{operators} \cE\bu = -\mu\Delta \bu - (\lambda+\mu) \nabla \nabla \cdot \bu;~~~~ Ap = -\nabla \cdot[k\nabla p]=-k\Delta p.\end{equation} Invoking the smoothness of the domain $\Omega$ (and standard elliptic regularity), we characterize the domain $$\mathcal D(\mathcal E) \equiv \mathbf H^2(\Omega)\cap \bV,$$ which corresponds to homogeneous Dirichlet conditions for the elastic displacements. Similarly, we take $$\mathcal D(A) \equiv \{ p \in H^2(\Omega)\cap L^2_0(\Omega)~:~\gamma[\nabla p\cdot \mathbf n]=0\}.$$ Here, $\cE$ realizes an isomorphism in two contexts: $\mathcal D(\cE) \to L^2(\Omega)$ and $\bV \to \bV'$, where $\cE^{-1}$ is interpreted respectively (i.e., through its natural coercive bilinear form $e(\cdot,\cdot)$) \cite{show1}. 
Similarly, $A:\mathcal D(A) \to L^2(\Omega)$ or $V\to V'$ is an isomorphism. See \cite{bmw} for more discussion.

Lastly, we define the nonlocal, zeroth order {\em pressure-to-dilation} mapping as follows: \begin{equation} B\equiv -\nabla \cdot \cE^{-1}\nabla. \end{equation} As a mapping on $L^2(\Omega)$, $B$ is central to many abstract analyses of Biot dynamics \cite{bw,bgsw,cao2}. We state its relevant properties as a lemma coming from \cite{bmw} in the specific context of $L^2_0(\Omega)$ and $V = H^1(\Omega)\cap L^2_0(\Omega)$:
 \begin{lemma}\label{Binvert}
The operator $B \in \mathscr L(L^2_0(\Omega)) \cap \mathscr L (V)$. $B$ is an isomorphism on $L^2_0(\Omega)$ and is injective on $V$.
Finally, we have that $B$ is a self-adjoint, monotone operator when considered on $L^2_0(\Omega)$.
\end{lemma}

Finally, we conclude with a definition of weak solutions for \eqref{PVEsystem} which will be valid in all parameter regimes, and is consistent with the abstract definition given in the Appendix for \eqref{ImplicitCauchy}. We note that such a definition encompassing all regimes does not seem to have appeared in the literature to date.
Recall that the fluid content is given by ~$\zeta \equiv c_0p+\alpha \nabla \cdot \bu + \delta_2\nabla \cdot \bu_t$, which depends on the nonnegative paremeters $c_0$ and $\delta_2$.
\begin{definition}\label{weaksol} Let $c_0, \delta_1,\delta_2 \ge 0$. 
We say that $(\bu,p) \in L^2(0,T;\bV \times V)$, such that $\delta_1\bu\in H^1(0,T;\bV)$ and $\zeta_t\in L^2(0,T;V')$, is a weak solution to problem \eqref{PVEsystem} if: 
\begin{itemize}
	\item For every pair of test functions $(\bv,q)\in\bV\times V$, the following equality holds in the sense of $\mathcal D'(0,T)$:
	\begin{align}\label{WeakFormEq}
		e(\bu,\bv)+\delta_1\frac{d}{dt}e(\bu,\bv)
		+ (\alpha\nabla p,\bv)_{L^2(\Omega)}
		+\frac{d}{dt}(\zeta,q)_{L^2(\Omega)}+a(p,q)
		=\langle \bF,\bv \rangle_{\bV'\times \bV}
		+\langle S,q \rangle_{V'\times V}.
	\end{align}
	\item The initial conditions $\zeta(0)=d_0$ and $\delta_1\bu(0)=\bu_0$ are satisfied in the sense of $C([0,T];V')$ and $C([0,T];\bV')$, respectively. 
\end{itemize}

\end{definition}
There are many notions of ``stronger" solutions to Biot-type systems in the literature (see \cite{show1}, for instance). To avoid confusion with notions of {\em strong} or {\em classical} solutions coming from other references, any notion of a {stronger solution} will be discussed here in the sense of {weak solutions with additional regularity}. Depending on the regularity of the sources in given cases, we will comment on when PDEs hold in a point-wise sense.

\subsection{Review of Biot Solutions: $\delta_1=\delta_2=0$}\label{secbiot}
We begin with a discussion of classical Biot dynamics in order to establish a baseline for comparison with our results below on poro-visco-elastic dynamics.
Consider now the quasi-static Biot dynamics---in the absence of visco-elastic effects---given in the operator-theoretic form by 
\begin{equation}\label{bioteq} \begin{cases}
\cE\bu+\alpha \nabla p = \mathbf F \in H^1(0,T;\bV') \\ 
[c_0p+  \alpha \nabla \cdot \bu]_t+Ap=S \in L^2(0,T;V')\\
[c_0p+  \alpha \nabla \cdot \bu](0) = d_0 \in V'.
\end{cases}
\end{equation}
Although there are established references (such as \cite{showmono,show1}) that discuss linear Biot solutions, as well as more recent papers (such as \cite{bmw}), we provide here a direct discussion of Biot solutions. Namely, we present the regularity of solutions as a function of the data, and clearly state the associated a priori estimates. While Theorems \ref{th:main2} and \ref{strongBiot} are not novel results, it is valuable to present them in this way to provide context with our visco-elastic results for $\delta_1>0$ below; we also provide brief proof sketches for $\delta_1=0$ for completeness. 

From the established theory for implicit, degenerate equations (discussed in the Appendix), one seeks weak solutions in the class $$(\bu,p) \in L^2(0,T;\bV\times V).$$
Formally solving the elasticity equation $a.e.~t$~ as ~$\bu=-\alpha \cE^{-1}\nabla p+\cE^{-1}\mathbf F$, and relabeling the source $$S~ \mapsto ~S-\nabla \cdot \cE^{-1}\mathbf F_t\equiv \widetilde S,$$ we obtain the reduced, implicit  equation:
\begin{equation}\label{reducedBiot}
[\mathcal Bp]_t+Ap=\widetilde S \in L^2(0,T;V'),~~[\mathcal Bp](0)=d_0 \in V',~~ \text{ where }  \mathcal B = (c_0I+\alpha^2B). 
\end{equation}

The above system can, in principle, degenerate if $c_0=0$ and $B$ has a non-trivial kernel \cite{bmw}---though this will not be the case here. Indeed, in this work, $\mathcal B$ is invertible on $L^2_0(\Omega)$ independent of $c_0\ge 0$.
\begin{remark} We  note that the temporal regularity of $\bF$ is directly invoked in the reduction step. Namely, to consider $\widetilde S$ as a given RHS for the abstract degenerate equation, we must require that $$\nabla \cdot \cE^{-1}\mathbf F_t \in L^2(0,T;\bV');$$ this provides  consistency with the original source, $S$. Additionally, to ``solve" the elasticity equation for $\bu$ (given $p$ and $\bF$) we will require $\bF \in \bV'$ ~$a.e. ~t$ to obtain $\bu \in \bV$. Smoother considerations below will require additional spatial regularity for $\bF$ and $\bF_t$. Regularity of $\bF_t$ is at issue for the analysis of Biot's dynamics \cite{bmw} and the analysis below. \end{remark}

To obtain the results below, one applies the general theory developed in \cite{auriault} and \cite{show1} for Biot's dynamics, and adapted recently in \cite{bw,bmw}. 

 \begin{theorem}\label{th:main2}  Let $d_0\in L_0^2(\Omega)$, $\bF \in H^1(0,T;\bV')$, $S \in L^2(0,T;V')$, and $c_0\geq 0$. Then there exists a unique weak solution with $(\bu,p)\in C([0,T];\bV)\times L^2(0,T;V)$ to \eqref{bioteq}. Moreover, any weak solution satisfies the following energy estimate:
\begin{align}\label{Thm2Estimate} 
{||\bu||_{L^{\infty}(0,T;\bV)}^2} + 
c_0||p||^2_{L^{\infty}(0,T;L^2(\Omega))}
+\|p\|^2_{L^2(0,T;V)} &
\\\nonumber
  \lesssim ||d_0||_{L^2(\Omega)}^2&+\|S\|^2_{L^2(0,T;V')} +C(T)||\bF||_{H^1(0,T;\mathbf V')}^2.
\end{align}
In all cases ($c_0 \ge 0$) the dynamics are {\em parabolic} in the sense that if $S \equiv 0$ and $\mathbf F \equiv 0$, we have:
\begin{align}
||Ap||_{L^{\infty}(0,T;L^2(\Omega))}+||\cE\bu||_{L^{\infty}(0,T;\mathbf L^2(\Omega))} + ||\nabla \cdot \left[\cE \bu\right] ||_{L^{\infty}(0,T;L^2(\Omega))} \lesssim \dfrac{||d_0||_{L^2(\Omega)}}{T},
\end{align}
to which elliptic regularity for $A$ and $\cE$ can then be applied (as in \cite{show1}).
\end{theorem}
\begin{proof}[Proof Sketch of Theorem \ref{th:main2}]
The theorem above can be obtained in two steps: First, weak solutions can be constructed directly (for instance through Galerkin method or by the theory  in the Appendix, e.g., \cite{bmw,bw,cao2,show1}.) The particular {\em constructed solution} will satisfy the energy identity \cite{bmw,showmono}. Then, using a classical argument involving the test function   $ \int_t p~ds$ in the reduced pressure equation \eqref{reducedBiot} (see \cite[pp.116--117]{showmono}), one  concludes that weak solutions are unique. \end{proof}
We mention that the issue of uniqueness is much more subtle in the case of time-dependent coefficients. See the detailed discussion in \cite{bmw}. 

\begin{remark} Above, we have the ability to {\em  specify only the quantity} $d_0 \in L^2_0(\Omega)$---rather than a pair $(\bu_0,d_0)$ or $(p_0,d_0)$, so that ~$c_0p_0+\nabla \cdot \bu_0 =d_0$. Indeed, given $d_0 \in L^2_0(\Omega)$ in this framework and recalling that $$\mathcal Bp(0)=[(c_0I+B)p](0)=c_0p_0+\nabla \cdot \bu_0$$ we have that
$$d_0 \in L_0^2(\Omega) \implies \mathcal Bp(0) \in L_0^2(\Omega) \implies p(0) \in L_0^2(\Omega) \implies \cE\bu(0) \in \mathbf V' \implies \bu(0) \in \bV.$$
In the case when the operator $B$ is not invertible on a chosen state space and $c_0=0$ (e.g. $L^2(\Omega)$), the issue can be more subtle. (See \cite{bmw} for more discussion, as well as the original papers \cite{show1,auriault}.) The equivalence of specifying $p_0$ and $d_0$  will not necessarily be available when $\delta_1>0$ and an additional time derivative is present in the equations.\end{remark}

We now briefly describe the notion of a smooth solution for the classical Biot dynamics above, when the data are smooth. These results can be obtained through elliptic regularity for $\cE$ and $A$ on $L^2(\Omega)$, and formal a priori estimates via the weak form, or via the implicit semigroup formulation as applied to Biot's dynamics in \cite[Theorems 3.1 and 4.1]{show1}. We first invoke the properties of $B$ in the context of \eqref{bioteq} to obtain the chain:
$$d_0 \in V \implies \mathcal Bp(0) \in V \implies p(0) \in V \implies \cE\bu(0) \in \mathbf L^2(\Omega)\implies \bu(0) \in \mathcal D(\cE).$$ 
Via the standard methodology for parabolic dynamics equation, choosing stronger initial data yields  additional regularity.
\begin{theorem}\label{strongBiot}
If $d_0 \in V$,  with $S \in H^1(0,T;V')$ and $\bF \in H^2(0,T;V')$, there exists a unique weak solution with the regularity:
$$p \in H^1(0,T; L_0^2(\Omega)) \cap L^{\infty}(0,T;V)~~\text{and}~~\bu \in H^1(0,T;\bV).$$

If, in addition, $S \in L^2(0,T; L^2(\Omega))$ and $\bF \in H^1(0,T;\mathbf L^2(\Omega))$, the above solution also has \eqref{bioteq} a.e. $t$ and $a.e.$ $\mathbf x$ and we obtain the additional regularity:
$$\bu \in L^2(0,T;\mathcal D(\cE)),~~p \in L^{\infty}(0,T;\mathcal D(A)).$$ 
Lastly, if we also assume $\bF \in L^{\infty}(0,T; \mathbf L^2(\Omega))$, then $\bu \in L^{\infty}(0,T;\mathcal D(\cE)).$
\end{theorem}
We note that,  solutions of higher regularity---for instance considering $d_0$ or $p_0 \in \mathcal D(A)$---can be considered; however, one must address commutators associated to boundary conditions encoded in $A$ and $B$. This can be seen, for instance, in attempting to test \eqref{reducedBiot} with $Ap_t$. 

For completeness, we provide the formal identities which give rise to the smooth solutions above.
\begin{proof}[Proof Sketch of Theorem \ref{strongBiot}]
Consider the reduced form of the Biot equation, 
$$[\mathcal Bp]_t+Ap=\widetilde S \in H^1(0,T;V'),$$
with $\mathcal B = [c_0I+\alpha^2B],$  as defined in Section \ref{opsec} and $$\widetilde S \equiv S-\nabla \cdot \cE^{-1}\mathbf F_t \in H^1(0,T;V').$$

Consider a smooth solution (as, for instance, for finite dimensional approximants) and test with $p_t$ to obtain:
\begin{align} \label{mod1}
\dfrac{k}{2}||\nabla p(T)||^2+\int_0^T(\mathcal Bp_t,p_t)dt =& ~ \frac{k}{2}||\nabla p(0)||^2+\langle\widetilde S(0) , p(0) \rangle_{V'\times V} \\ \nonumber
& + \langle \widetilde S(T), p(T)\rangle_{V'\times V} + \int_0^T\langle \widetilde S_t, p \rangle_{V'\times V} dt.
\end{align}
Alternatively, if we assume that $\widetilde S \in L^2(0,T;L^2_0(\Omega))$ (which follows from the additional assumptions above), then the identity is similar: 
\begin{align} \label{mod2}
\dfrac{k}{2}||\nabla p(T)||^2+\int_0^T(\mathcal Bp_t,p_t)dt =& ~ \frac{k}{2}||\nabla p(0)||^2+\int_0^T(\widetilde S,p_t)dt\end{align}
In both situations, the assumed regularity of the data is sufficient to estimate the RHS and obtain an estimate on $p$.

With regularity of the pressure $p$ in hand, we consider the full system in \eqref{bioteq} and formally differentiate the elasticity equation \eqref{bioteq}$_1$. This yields:
\begin{equation}\label{bioteq*} \begin{cases}
\cE\bu_t+\alpha \nabla p_t = \mathbf F_t \in L^2(0,T;\bV') \\ 
[c_0p+  \alpha \nabla \cdot \bu]_t+Ap=S \in L^2(0,T;L^2_0(\Omega))~~ \text{or} ~~H^1(0,T;V')
\end{cases}
\end{equation}
We can test the first equation by $\bu_t$ and the second by $p_t$ and add to obtain the identity:
\begin{align}
e(\bu_t,\bu_t)+c_0||p_t||^2+\dfrac{k}{2}\dfrac{d}{dt}||\nabla p||^2 = \langle \bF_t, \bu_t\rangle_{\bV'\times \bV} +(S,p_t )_{L^2(\Omega)},
\end{align}
where we have assumed the case $S \in L^2(0,T;L_0^2(\Omega))$---the appropriate modifications are clear for the other case, as in \eqref{mod1} and \eqref{mod2} above.
The RHS can be estimated, with the assumed regularities of $\bF_t$ and $S$. The additional regularities stated in the theorem are then read-off from the individual equations in \eqref{bioteq}.
\end{proof}
\subsection{Visco-elastic Cases of Interest} \noindent
In considering { full, linear poro-visco-elasticity}, we will take $\delta_1>0$, and retain the parameter to track terms which depend on it. We consider the independent cases:  \begin{itemize} \setlength\itemsep{.01cm} \item  compressible constituents $c_0>0$ and incompressible constituents $c_0=0$;
\item standard fluid content $\delta_2=0$ and adjusted fluid content $\delta_2>1$. \end{itemize} 
This yields four cases of interest for:
\begin{equation}\label{main1} \begin{cases}
\mathcal \cE\bu+\delta_1\cE\bu_t+\alpha \nabla p = \mathbf F \in H^1(0, T; \bV')\\ 
[c_0p+\alpha \nabla \cdot \bu+\delta_2\nabla \cdot \bu_{t}]_t+Ap=S  \in H^1(0,T;V').
\end{cases}
\end{equation}

In the analysis in the sequel, we will often require additional temporal regularity on the volumetric source $S$ and additional regularity for $\bF$ beyond what is specified above. It is natural, and analogous to \eqref{bioteq}, to take initial conditions of the form
\begin{equation}\label{naturalIC}[c_0p+\alpha \nabla \cdot\bu+\delta_2\nabla \cdot \bu_t](0)=d_0 \in V', ~\text{and}~ ~\delta_1\cE(\bu(0))\in \bV'.\end{equation}
However, we will discuss  initial conditions more precisely on a case by case basis. In fact, {a main feature of our subsequent analysis is in addressing this point}. Which initial quantities can be specified depends on the specific parameter regime ($\delta_1,\delta_2,c_0 \ge 0$), of course being mindful of possible {\em over-specification}. Owing to the {time-derivatives present in both equations}---in contrast to Biot's traditional equations \eqref{bioteq}---we are unable to circumvent the need to specify two initial quantities; however the relationship between them will be an interesting question to be addressed.

\vskip.2cm \noindent {\bf Summary of Initial Conditions}: We now provide a summary of  proper specifications of initial conditions, with justifications to follow in the appropriate sections. Of course, there are questions of scaling and regularity of these conditions. Such matters will translate into the sought-after notion of solution. Though a natural quantity is the fluid content, we relegate our summary to the two primal variables: $(p,\bu)$ with possible initial conditions $(p(0),\bu(0))$. Of course these are possibly related through the quantity $$d_0 = [c_0p+\alpha \nabla \cdot \bu+\delta_2\nabla \cdot \bu_t](0).$$
{The proper initial conditions} for \eqref{main1} are  given in the table below. 
We take $\delta_1>0$ and consider $c_0 \ge 0,~ \delta_2 \ge 0$.
\begin{center}
\begin{tabular}{|l | l | l | }
\hline
& $c_0=0$ & $c_0>0$   \\
\hline
$\delta_2=0$ &  $p(0)$ or $(p(0), \bu(0))$ & $(p(0),\bu(0))$ or $(p(0), p_t(0))$      \\
\hline
$\delta_2>0$  &    $p(0)$ or $(p(0), \bu(0))$  & $p(0)$ or $(p(0), \bu(0))$  \\
\hline
\end{tabular}
\end{center}

\begin{remark} There may be  physical restrictions about the permissibility of certain parameter combinations. For instance, in the application to biological tissues, when it is standard to take $\alpha=1$ and $c_0=0$ \cite{GGbook}, the parameter $\delta_2$ should be nullified \cite{MBE1,MBE2}. This is to say, the combination $\delta_1, \delta_2>0, ~c_0=0$ may not be physically relevant, however, in this mathematically-oriented work we  accommodate all parameter combinations and describe the features of the resulting dynamics. \end{remark}

\section{Poro-Visco-elastic System, Reduction, and Solutions}\label{viscosec}
\noindent Section \ref{viscosec} constitutes the main thrust of the paper. We will now consider the linear poro-visco-elastic system, as presented in \eqref{main1}, with $\delta_1>0$. We note that the boundary conditions are embedded in the operators $A$ and $\cE$.
\subsection{Outline and Discussion of Main Results} \noindent
Section \ref{viscosec} is divided as follows: First, we consider the traditional fluid content in the model (taking $\delta_2=0$) in Section \ref{thissec} and conditioning on the values of the storage coefficient $c_0 \ge 0$ in the contained subsections. Subsequently, in Section \ref{thiss}, we consider the model with modified fluid content ($\delta_2>0$).  In addition to analyzing the full system, we will reformulate the model abstractly to apply established results and obtain well-posedness of the dynamics in a variety of  functional frameworks. 
{In each case described in Section \ref{viscosec}, we provide:}
\begin{itemize}  \setlength\itemsep{.05cm}\item a discussion of the dynamics, \item a state reduction, \item a discussion of initial conditions, \item and a well-posedness theorem with estimates. \end{itemize}

When it is instructive, we provide corresponding  estimates on solutions and describe the resulting constructions of  solutions. {While the abstract results and main techniques that we employ are not novel to this paper, (i) {\em the abstract problem formulation is}, as well as (ii) {\em the application of these abstract results to the linear poro-visco-elastic model}. Such results on poro-visco-elastic systems have not appeared in the literature} to the knowledge of the authors. 

The novel results we obtain here are now briefly described. 
 \begin{itemize}  \setlength\itemsep{.05cm}
\item Theorem \ref{Thm-DWave-Energy1} utilizes the full formulation of the linear poro-visco-elastic dynamics and provides clear a priori estimates on solutions, not distinguishing between cases with the storage coefficient $c_0\ge0$. 

\item Theorem \ref{Thm-DWave-Energy2} gives a well-posedness result when $c_0>0$ that is obtained through a priori estimates without the use of the semigroup framework. Theorem \ref{dampedSG} obtains a well-posedness result and utilizes the semigroup framework. Different assumptions on initial conditions provide different outcomes in these aforementioned results. Subsequently, a semigroup decay result is presented in Theorem \ref{decay}. 

\item A similar sequence is repeated, on different spaces, when $c_0=0$; this case is referred to as the ODE case, for reasons explained below. Theorem \ref{ODEt1} provides well-posedness of solutions, Theorem \ref{ODEt2} provides a detailed description of the regularity of solutions, and Theorem \ref{ODEdecay} gives explicit exponential decay, even in this ODE setting, obtained through direct estimates. 

\item In the case where $\delta_2>0$ (modified fluid content), we have only one new theorem, Theorem \ref{finth}. This is owing to the fact that the abstract structure of this problem reduces to that of the traditional elastic Biot system, to which the previous results are then applied. The novel contribution here, then, is to demonstrate how the system is reduced in this fashion.
\end{itemize}
 
\subsection{Case 1: Traditional Fluid Content, $\delta_2$=0}\label{thissec}
Take $\delta_2=0$  in \eqref{main1}. We begin with a formal discussion of weak solutions for the full system, and then proceed to the abstract system reduction. After these discussions, we  proceed to rigorous statements concerning solutions.

\subsubsection{Motivating Discussion}
Let us begin by describing the energy identity for the {\em full} dynamics, before any system reduction is made.
Recall the system full dynamics under consideration:
\begin{equation}\label{main1**} \begin{cases}
\mathcal \cE\bu+\delta_1\cE\bu_t+\alpha \nabla p = \mathbf F \in H^1(0, T; \bV')\\ 
[c_0p+\alpha \nabla \cdot \bu]_t+Ap=S  \in H^1(0,T;V').
\end{cases}
\end{equation}
From this, we will have (following from \cite{bgsw}) the a priori estimate on solutions:
\begin{align}\label{energyest}
||\bu||^2_{L^{\infty}(0,T;\bV)}+c_0||p||^2_{L^{\infty}(0,T;L^2(\Omega))}+\delta_1||\bu_t||^2_{L^2(0,T;\bV)}+||p||^2_{L^2(0,T;V)} & \\\nonumber
\lesssim c_0||p_0||^2 + ||\bu_0||_V^2& +||S||^2_{L^2(0,T;V')} + ||\bF||^2_{L^2(0,T; \mathbf L^2(\Omega))}.
\end{align}
Several comments are in order:
\begin{itemize} 
\item We need not invoke any {\em additional} regularity of the sources $S$ or $\bF$ beyond those appearing in \eqref{main1**} and \eqref{energyest}; indeed we will  do this later.
\item Even in the estimate above, one can replace the norm on $\bF$ as follows: $||\bF||^2_{L^2(0,T; \mathbf L^2(\Omega))} \mapsto ||\bF||_{H^1(0,T; \bV')}.$ In this case, the constant on the RHS corresponding to $\lesssim$ becomes dependent on time, i.e., $$\text{LHS} \le C(T)\text{RHS}~\text{ in \eqref{energyest}. }$$ 
\item Note, from the RHS, $d_0=[\nabla \cdot \bu+c_0p](0)$ does not explicitly appear; on the other hand, any two of the three of $\bu(0)=\bu_0$, $p(0)=p_0$, or $d_0$ may be specified, with the third obtained immediately thereafter.
\end{itemize}

Now, we proceed to obtain an abstract reduction of \eqref{main1**}, in line with what was done for \eqref{reducedBiot} for the purely elastic dynamics ($\delta_1=\delta_2=0$). This is the central insight in the analysis of these poro-visco-elastic dynamics. 

From the displacement equation \eqref{main1**}$_1$, we have: \begin{equation}\label{StructureODE}
	\delta_1\bu_t+\bu=\cE^{-1}\mathbf F-\alpha \cE^{-1}(\nabla p).
\end{equation}
{This equation can be explicitly solved for $\bu$ as an ODE in $t$ for $a.e.$~ $\mathbf x$} (see Lemma \ref{ODEsolve} below).
We may then differentiate the pressure equation in time: $$c_0p_{tt}+\alpha \nabla\cdot \bu_{tt}+Ap_t=S_t \in L^2(0,T;V').$$
We  then rewrite $\nabla \cdot \bu_{tt}$ through the time derivative of \eqref{main1}$_1$:
\begin{align*}
\delta_1\nabla \cdot \bu_{tt} =  -\nabla \cdot \bu_t+\nabla \cdot \cE^{-1}(\bF_t)+\alpha Bp_t =\alpha^{-1}[c_0p_t+Ap]-\alpha^{-1}S+\nabla \cdot \cE^{-1}(\bF_t)+\alpha Bp_t.
\end{align*}
Recalling the definition $B=-\nabla \cdot \cE^{-1}\nabla$, and taking \begin{equation} \widehat{ S}=\delta_1^{-1}[S-\alpha \nabla \cdot \cE^{-1}(\bF_t)]+S_t,\end{equation} we  have obtained a reduced pressure equation in this case:
\begin{equation} \label{strongdamp} c_0p_{tt}+\left[A+\delta_1^{-1}{(c_0I+\alpha^2B)}\right]p_t+\delta_1^{-1}Ap=\widehat S.\end{equation}

We are now in a position to make several salient observations about linear poro-visco-elastic dynamics: \begin{itemize} \item When $c_0>0$, we observe a {\em strongly damped} hyperbolic-type equation \cite{redbook,GilbertAcoustic}(and references therein). The damping operator $\mathbf D$ is given by \begin{equation}\label{dampop} \mathbf D \equiv A+\delta_1^{-1}{(c_0I+\alpha^2B)}=A+\delta_1^{-1}\mathcal B.\end{equation} It can be interpreted as $\mathbf D: V \to V'$ or from $\mathcal D(A) \to L^2(\Omega)$. This operator is nonlocal but has the property of being $A$-bounded in the sense of \cite[p.17]{trig2} (see also \cite{redbook}). Roughly, $\mathbf D$ being $A$-bounded means that $\mathbf D$ acts as $A$ does in its dissipative properties. 
\item We provide an explicit definition for weak solutions below for the reduced problem \eqref{strongdamp}.
\item There is clear singular behavior in the equation in the visco-elastic parameter $\delta_1\searrow 0$.
\item To obtain the reduced equation, we have differentiated the fluid source, $S$. We then inherit the requirement  that $S \in H^1(0,T;V')$ to make use of the reduced formulation. Consequently, we again need $\bF \in H^1(0,T;\mathbf \bV')$ when invoking the {abstract reduced form} in \eqref{strongdamp}. 
\item It is not obvious at this stage what the connection is between the primally specified initial quantities, $\bu(0)$ and  $d_0$, and those  standard ones for the strongly damped wave equation, $p(0)$ and $p_t(0)$; we resolve this below, with distinct theorems for these two different frameworks.
\item The formal energy identity  for the  reduced dynamics on $t\in[0,T]$ can be written:
\begin{align}\label{EED1C0}
	\frac{1}{2}\left [c_0\|p_t(T)\|^2 +\frac{1}{\delta_1}a(p(T),p(T))\right]&+\int_0^T\left[a(p_t,p_t)+
	\delta_1^{-1} \big(\mathcal Bp_t,p_t\big)_{L^2(\Omega)}\right]dt
	\nonumber \\ =&~\frac{1}{2}\left [c_0\|p_t(0)\|^2 +\frac{1}{\delta_1}a(p(0),p(0))\right ]+\int_0^T\langle \widehat{S},p_t\rangle_{V' \times V}dt.
\end{align}
We shall reference this later.
\end{itemize}

In the next subsection, we present a well-posedness and regularity theorem (in two parts) which is independent of $c_0 \ge 0$. In Theorem \ref{Thm-DWave-Energy1} we approach the problem through the full formulation and  provide  a priori estimates and various regularity assumptions on the data.  Secondly, in the  subsections that follow, we will consider $c_0>0$ and $c_0=0$ separately. For the full system with $c_0>0$, we have Theorem \ref{Thm-DWave-Energy2}. For the reduced formulation, we provide Theorem \ref{dampedSG}, which is achieved through the second-order semigroup theory, and requires specification of both $(p(0), p_t(0))$; from the obtained solution, we infer regularity about the ``natural" initial quantities. Following these theorems, we may compare the resulting regularity of the produced solutions. 

\subsubsection{General Well-posedness Result: $c_0\ge 0$}

\begin{theorem}\label{Thm-DWave-Energy1}
Consider $c_0 \ge 0$ in \eqref{main1**}. Let  $S\in L^2(0,T;V')$, $\bF\in H^1(0,T;\bV')\cup L^2(0,T;\mathbf L^2(\Omega))$. Take $\bu_0 \in \bV$ and $c_0p_0 \in L_0^2(\Omega)$. 
\vskip.1cm
\noindent [Part 1] Then there exists unique weak solution, ~$p \in L^2(0,T;V)$,~ $\bu \in H^1(0,T;\bV)$, and ~$[c_0p+\alpha \nabla \cdot \bu]_t \in L^2(0,T;V')$, as in Definition \ref{weaksol}. This solution is of finite energy, i.e., the identity \eqref{energyest} holds.

From the above facts, we also infer: 
\begin{itemize} \setlength\itemsep{.01cm}
\item $\bu \in C([0,T];\bV))$,
 \item $c_0p_t \in L^2(0,T;V')$,
 \item $c_0p \in C([0,T]; L_0^2(\Omega)).$ 
 \end{itemize}
\vskip.1cm
\noindent [Part 2] In addition to the previous assumptions,  take $p_0 \in V$ and $\bu_0 \in \mathcal D(\cE)$, and assume that $\bF \in H^1(0,T;\mathbf L^2(\Omega))$ and $S \in H^1(0,T;V')\cap L^2(0,T;L_0^2(\Omega))$. Then the above weak solution has the additional regularity that $$p \in H^1(0,T;L^2(\Omega)) \cap L^{\infty}(0,T;V) \cap L^2(0,T;\mathcal D(A))~~ ~\text{ and }~~~ \bu_t \in L^{\infty}(0,T;\bV).$$ The resulting solutions satisfy system \eqref{main1**} a.e. $\mathbf x$ a.e. $t$. 
Furthermore, if $\bF \in L^2(0,T;\bV)$ also, then we have  $\cE\bu \in H^1(0,T;\mathbf V)$ and $\bu \in H^2(0,T;\bV)$ in addition.
\end{theorem}

 In Part 1 of the Theorem above, there is not enough regularity to infer any spatial regularity of $p_t(0)$ from the equation. 
In Part 2, however, we can read-off the regularity of $p_t(0) \in V'$ (when $c_0>0$)  from the pressure equation a posteriori as follows: since $S(0)\in V'$ is defined for $S \in H^1(0,T;V')$:
$$c_0p_t(0) =S(0)-Ap(0)-\alpha \nabla \cdot \bu_t(0) \in V'.$$
We also obtain an initial condition for $\bu_t(0) \in \mathcal D(\cE)$ in Part 2 of the theorem via the elasticity equation and the regularity of $p_0, \bu_0,$ and $\bF$.

\begin{proof}[Proof of Theorem \ref{Thm-DWave-Energy1}] The construction of solutions in Part 1 of the theorem follows a standard approach via the a priori estimate (the baseline energy inequality) in \eqref{energyest}. (For instance, see the construction given through full spatio-temporal discretization given in \cite{bgsw}.) Uniqueness is obtained straightforwardly, as, for a weak solution, the function $\bu_t \in L^2(0,T;\bV)$ is a permissible test function in the elasticity equation \eqref{main1**}$_1$. (See the analysis in \cite{bgsw}.) This implies that {\em any weak solution} satisfies the energy inequality \eqref{energyest}. As the problem is linear, uniqueness then follows. 

To obtain Part 2, we point to the requisite a priori estimate for higher regularity. This estimate can be obtained in the discrete or semi-discrete framework (i.e., on Galerkin approximants, as in \cite{bw}), and the constructed solution  satisfies the resulting estimate. Uniqueness at the level of weak solutions remains. 
To obtain the a priori estimate, differentiate \eqref{main1**}$_1$ in time and test with $\bu_{t}$, then test \eqref{main1**}$_2$ with $p_t$. This produces the formal identities
\begin{align}
\dfrac{\delta_1}{2} \dfrac{d}{dt} e(\bu_t,\bu_t)+e(\bu_t,\bu_t)+\alpha(\nabla p_t,\bu_t)=&~( \bF_t, \bu_t) \\
c_0||p_t||^2+\alpha (\nabla \cdot \bu_t,p_t)+\dfrac{k}{2}\dfrac{d}{dt}||\nabla p||^2 =&~ ( S, p_t ).
\end{align}
We note that, from the second identity above, we can treat the term $(S,p_t)$ as an inner product (and absorb $p_t$ on the LHS) when $c_0>0$. However, to have a result which is independent of $c_0\ge 0$, we relax the regularity below. Doing so, adding the two identities, and integrating in time, we obtain:
 \begin{align*} 
\dfrac{\delta_1}{2}e(\bu_t(T),\bu_t(T))&+\dfrac{k}{2}||\nabla p(T)||^2+\int_0^T[e(\bu_t,\bu_t)+c_0||p_t||^2]dt \\ = &~\dfrac{\delta_1}{2}e(\bu_t(0),\bu_t(0))+\dfrac{k}{2}||\nabla p_0||^2 
 +\int_0^T ( \bF_t, \bu_t) dt 
 -\int_0^T\langle S, p\rangle_{V'\times V} dt + \langle S,p\rangle_{V'\times V}\Big|_{t=0}^{t=T}.
\end{align*}

\noindent The RHS and the data $S$, $\bF$, and $p_0$ have appropriate regularity to control the LHS. We note, of course, that at $t=0$ we have:
$$\cE(\bu_t(0))=-\alpha \nabla p_0-\cE(\bu_0)+\bF(0);$$ which is bounded in $\bV'$. Since $\cE: \bV \to \bV'$ is identified by the Riesz Isomorphism, this gives that
$$||\bu_t(0)||^2_{V} \lesssim ||p_0||_V^2+||\bu_0||_{\bV}^2+||\bF||^2_{H^1(0,T;\bV')}.$$
From the above, we obtain that solutions are bounded in the sense of $\bu_t \in L^{\infty}(0,T;\bV)$ and $p \in L^{\infty}(0,T;V) \cap H^1(0,T;L_0^2(\Omega))$. The remaining statements on regularity are read-off from the equations \eqref{main1**} for the data as prescribed respectively in the statement of Theorem \ref{Thm-DWave-Energy1}. 
\end{proof}
\begin{remark}
An alternative way to obtain the same result for more regular solutiosn is to invoke the test function $\cE(\bu_t)$ in the displacement equation \eqref{main1**}$_1$. Noting that the divergence operator commutes with the Laplacian, we observe ~$\bu \in H^1(0,T;\mathcal D(\cE))$, ~$p\in L^{\infty}(0,T;V)\cap L^2(0,T; \mathcal D(A))$, having specified initial conditions ~$p_0\in V$, $\bu_0\in \mathcal D(\cE).$
\end{remark}

\subsubsection{Compressible Constituents, $c_0>0$} \label{cpos} 
The equation \eqref{strongdamp} above, with $c_0>0$, is hyperbolic-like with strong damping; this renders the  entire system parabolic \cite{redbook,trig2}, with associated parabolic estimates. The damping operator $\mathbf D$ defined in \eqref{dampop} is $A$-bounded in the sense of \cite{trig2}. Equations of such type can also arise in acoustics---see e.g. \cite{GilbertAcoustic}. We elaborate below through several additional theorems. Before doing so, let us provide a clear definition of weak solutions to the reduced wave equation \eqref{strongdamp}. Namely, when $c_0>0$, we define $p$ to be a weak solution to 
$$c_0p_{tt}+\left[A+\delta_1^{-1}{\mathcal B}\right]p_t+\delta_1^{-1}Ap=\widehat S,$$
with $ \widehat{ S}=\delta_1^{-1}[S-\alpha \nabla \cdot \cE^{-1}(\bF_t)]+S_t$ and $\mathcal B = c_0I+\alpha^2B$,
to mean:
\begin{definition}[Weak Solution of Reduced, Strongly-Damped Wave]\label{df:weaksolution}
Let {$\widehat S\in L^2(0,T;V')$}. A weak solution of \eqref{strongdamp} is a function
~$p\in  H^1(0,T; V)\cap H^2(0,T; V')$
such that for a.e. $t>0$ and all $q \in V$, one has
\begin{equation}\label{weakform}
\langle c_0 p_{tt},q \rangle_{V'\times V}+ \langle \mathbf Dp_t,q\rangle_{V'\times V} + \delta_1^{-1} a(p,q) = \langle \widehat S,q\rangle_{V'\times V}, \end{equation}
where we interpret $\mathbf D = A+\delta_1^{-1}{(c_0I+\alpha^2B)}: V \to V'$ through the properties of the operators $A$ and $B$ given in Section \ref{opsec}.
\end{definition}

We  point out that the requirements that $S \in H^1(0,T;V')$ and $\bF \in H^1(0,T;\bV')$ are sufficient to guarantee that $\widehat S \in L^2(0,T;V')$.

In anticipation of the use of semigroup theory applied to the abstract presentation of the dynamics in \eqref{strongdamp} as a {\em wave-type equation}, we now address the following question: \begin{quote} What regularity can be obtained from the system when specifying, as an initial state, $p_t(0) = p_1$?\end{quote} The next theorem addresses this  through a priori estimates on the full system \eqref{main1**}, before we move to the semigroup theory for \eqref{strongdamp} in the later Theorem \ref{dampedSG}.

\begin{theorem}\label{Thm-DWave-Energy2}
	Let $S\in H^1(0,T;L_0^2(\Omega))$ and $\bF\in H^1(0,T;\bV')$. Consider initial data of the form ~$p(0) = p_0\in V$ and {$ p_t(0) =p_1 \in L_0^2(\Omega)$}. 
	
	Then there exists unique, finite energy {weak solution $p$, i.e., in the sense of Definition \ref{df:weaksolution} to \eqref{strongdamp}}, with the identity \eqref{EED1C0} holding {in $\mathscr D'(0,T)$}.
	 
Assuming, in addition, that $\bu_0\in \mathcal D(\cE)$ and {$\bF\in L^2(0,T;\mathbf L^2(\Omega))$}, one obtains a unique weak solution to the full system \eqref{main1**} in the sense of Definition \ref{weaksol}.The following additional statements hold:
	\begin{itemize}
		\item The unique solution $\bu$  to \eqref{StructureODE} has $\bu\in H^1(0,T;\mathcal D(\cE))$.
		\item $p\in L^{\infty}(0,T;\mathcal D(A))$.
		\item $S\in L^2(0,T;V) ~\implies~ Ap\in L^2(0,T;V)$.
	\end{itemize}
\end{theorem}
\begin{proof}[Proof of Theorem \ref{Thm-DWave-Energy2}]
	Since \eqref{strongdamp} has the form of the strongly damped wave equation, the construction of solutions is standard and we omit those details. We suffice to say, solutions can be  constructed via the Galerkin method, and  approximants satisfy energy identity \eqref{EED1C0}. Therefore, one obtains a weak solution as weak/weak-$*$ limits of the approximations. Moreover, uniqueness follows directly from the energy identity and linearity of the system---the regularity of $p_t$ is sufficient for it to be used as a test function for an arbitrary weak solution (unlike the case for the undamped wave equation). 
	
	The displacement $\bu$ is recovered by using obtained regularity of $p$ and solving ODE \ref{StructureODE} in time,  in the space $\mathcal D(\cE)$ (see Lemma \ref{ODEsolve} below).
	Additional regularity of $p$ is obtained by using equation \eqref{main1**}$_2$:
	$$
	A p= S-c_0p_t-\alpha\nabla\cdot \bu_t\in L^{\infty}(0,T;L^2(\Omega)).
	$$
	The last statement of the theorem also follows by noticing $p_t\in L^2(0,T;V)$ and $\nabla\cdot \bu_t\in L^2(0,T;V)$.
\end{proof}

Here we emphasize that, in the original formulation, our main poro-visco-elastic equations, \eqref{main1}, {\em it is not necessary to specify $p_t(0)$}. However, $p_t(0)$ can be formally obtained from $\bu_0$ and $d_0$ as we now describe. Let us consider $$d_0 = c_0p(0)+\nabla \cdot \bu(0) \in V.$$ And, correspondingly, assume that $\bF \in H^1(0,T;\mathbf L^2(\Omega))$ and take $\cE\bu(0) \in \mathbf L^2(\Omega)$ to be fully specified. As $\cE : \mathcal D(\cE) \to \mathbf L^2(\Omega)$ is an isomorphism, this provides $\bu(0) \in \mathbf H^2(\Omega)\cap \bV$, and we can back-solve from $d_0 \in V$ to obtain $p(0) \in V$, providing $\nabla p(0) \in L^2(\Omega)$. 
Then, again from \eqref{main1}$_1$, we read-off $\cE[\bu+\delta_1\bu_t](0) \in L^2(\Omega)$ and infer that $\bu_t(0) \in \mathbf H^2(\Omega)\cap \bV$ since $$\delta_1\cE \bu_t(0)=\bF(0)-\alpha\nabla p(0)-\cE\bu(0) \in \mathbf L^2(\Omega).$$ Finally, $p_t(0)$ can be read-off from the pressure equation, when the time trace $S(0) \in L^2(\Omega)$ is well-defined:
\begin{equation}
c_0p_t(0)=S(0)-{Ap(0)}-\alpha \nabla \cdot \bu_t(0) \in {{V'}}.
\end{equation}

Thus, we observe that the energy methods (and standard solutions, as in Theorem \ref{Thm-DWave-Energy1} Part 2) applied to the original system gives a different (lower) regularity of solutions than that obtained in Theorem \ref{Thm-DWave-Energy2}. Thus, by prescribing $p_t(0) \in L^2_0(\Omega)$ (instead of prescribing $\bu_0$) and invoking the wave structure, we obtain an improved regularity result in Theorem \ref{Thm-DWave-Energy2}.

Now, we proceed to invoke the semigroup theory for second-order abstract equations with strong damping. Our primary semigroup reference will be \cite{pazy}, and for the strongly damped wave equation \cite{trig3,trig2}. In this case, our damper $\mathbf D$ is $A$-bounded. This is typically written as $A \le \mathbf D \le A$, which we rewrite here as: There exists appropriate constants such that 
$$(Aq,q)_{L^2(\Omega)} \lesssim (\mathbf Dq,q)_{L^2(\Omega)} \lesssim (Aq,q)_{L^2(\Omega)},~~\forall q \in \mathcal D(A).$$
We do not provide an in-depth discussion of the correspondence of the existence of a $C_0$-semigroup on a given state space and associated solutions, instead referring to the  \cite[Chapter 4]{pazy}. 

We now provide the framework for Theorem \ref{dampedSG}.
\begin{itemize}
\item The strongly damped wave equation in \eqref{strongdamp} has a first order formulation considering the state $$y=[p, p_t]^T \in Y \equiv V \times L^2_0(\Omega)$$ written as
\begin{equation} \label{SGform} \dot y = \begin{bmatrix} 0 & I \\ -[\delta_1c_0]^{-1}A & -c_0^{-1}\mathbf D \end{bmatrix} y+\mathscr F,~~y(0)=[p_0, p_1]^T.\end{equation}
\item The operator $\mathbb A \equiv \begin{bmatrix} 0 & I \\ -[c_0\delta_1]^{-1}A & -c_0^{-1}\mathbf D \end{bmatrix}$ is taken with domain 
\begin{equation}\label{SGform2}\mathcal D(\mathbb A) \equiv \mathcal D(A) \times \mathcal D(A),\end{equation} and $\mathbf D$ as given in \eqref{dampop} with
$\mathscr F \equiv [0,c_0^{-1}\widehat S]^T$. 
\end{itemize}

The theorem below will provide the existence of a semigroup $e^{\mathbb A t} \in \mathscr L(Y)$. In this context, we will obtain a solution $y(t) = e^{\mathbb At}y_0$ to the first order formulation in two standard contexts: 
\begin{itemize} \item When $y_0 \in \mathcal D(\mathbb A)$, the resulting solution lies in $C^1((0,T];Y)\cap C^0([0,T];\mathcal D(\mathbb A))$ and satisfies \eqref{SGform} pointwise.
\item When $y_0 \in Y$, the resulting solution is $C^0([0,T];Y)$ and satisfies a time-integrated version of \eqref{SGform}; these solutions are sometimes called {\em generalized} or {\em mild} \cite{pazy}, and are, in fact, $C^0([0,T];Y)$-limits of solutions from the previous bullet. Namely, we can approximate the data $y_0 \in Y$ by $y_0^n \in \mathcal D(\mathbb A)$, and obtain the generalized solution as a $C^0([0,T];Y)$-limit of the solutions emanating from the $y_0^n$.
\item It is standard (in this linear context) to obtain {\em weak solutions} (in the sense of Definition \ref{df:weaksolution}) by considering smooth solutions emanating from $y_0^n \in \mathcal D(\mathbb A)$ as approximants. 
\end{itemize} 

\begin{theorem}
[Damped Semigroup Theorem]  \label{dampedSG} In the  framework of \eqref{SGform}--\eqref{SGform2}, with $\delta_1,c_0>0$, the  operator $\mathbb A$ generates a $C_0$-semigroup of contractions $e^{\mathbb At} \in \mathscr L(Y)$, which is {\em analytic} (in the sense of \cite[Chapter 2.5]{pazy}) on $Y\equiv V\times L^2_0(\Omega)$. This is to say:
\begin{itemize} \item For $[p_0,p_1]^T \in Y$ and $\widehat S \in L^2(0,T; L^2_0(\Omega))$, we obtain a unique (generalized) solution $[p(\cdot),p_t(\cdot)]^T \in C([0,T]; Y)$ (as described above). 
\item For $[p_0,p_1]^T \in \mathcal D(A) \times \mathcal D(A)$ and $\widehat S \in H^1(0,T;L^2_0(\Omega))$, we obtain a unique solution $[p(\cdot),p_t(\cdot)]^T \in C([0,T]; \mathcal D(A) \times \mathcal D(A)) \cap C^1((0,T]; Y))$ that satisfies the system in a point-wise sense. \end{itemize}

Lastly, one may select the state space $Z = L_0^2(\Omega) \times L_0^2(\Omega)$ with $\mathcal D(\mathbb A)$ the same as before. In this context, $\mathbb A$ again generates a $C_0$-semigroup $e^{\mathbb At} \in \mathscr L(Z)$. This semigroup is again analytic, though it is not a contraction semigroup. In this case, with $[p_0,p_1]^T \in (L_0^2(\Omega))^2$ we obtain solutions in the sense of $C^0([0,T];Z)$.

\end{theorem}
\begin{proof}[Proof of Theorem \ref{dampedSG}] The proof of this theorem follows immediately from the application of \cite[pp.292--293]{redbook} to the present framework to obtain the semigroup.  For the case of taking the state space $Z$, see also \cite{trig3}. In passing from the semigroup to solutions---taking into account the inhomogeneity $\mathscr F$---we invoke \cite[Chapter 4.2]{pazy}.
\end{proof}
We note that $S\in H^1(0,T;L^2_0(\Omega))$ and $\bF\in H^1(0,T;\bV')$ implies $\widehat S\in L^2(0,T;L^2_0(\Omega))$. Moreover, stronger assumptions $S\in H^2(0,T;L^2_0(\Omega))$ and $\bF\in H^2(0,T;\bV')$ imply $\widehat S \in H^1(0,T;L^2_0(\Omega))$.

We can say more, since the strongly damped wave equation is known to be exponentially stable. 
\begin{theorem}[Exponential Decay of Solutions]\label{decay}
Consider the above framework in \eqref{SGform}--\eqref{SGform2}, and take $\bF\equiv \mathbf 0$ and $S \equiv 0$ in \eqref{main1**} (so $\mathscr F \equiv [0,0]^T$). Consider $c_0,\delta_1>0$. Then the analytic semigroup $e^{\mathbb At}$ generated by $\mathbb A: Y \supset \mathscr D(\mathbb A) \to Y$ is uniformly exponentially stable. That is, there exists $\gamma, M_0, M_k>0$ (each depending on $c_0,\delta_1>0$) so that: 
\begin{equation}\label{decayestimate}
||e^{\mathbb At}||_{\mathscr L} \le M_0e^{-\gamma t},~~~\text{and, more generally,}~~~||\mathbb A^ke^{\mathbb At}||_{\mathscr L} \le M_kt^{-k}e^{-\gamma t},~~t \ge 0, ~k \in \mathbb N.
\end{equation}
\end{theorem}
\begin{proof}[Proof of Theorem \ref{decay}]
As the above semigroup is analytic, and the requisite spectral criteria are satisfied by the operator $\mathbb A$, exponential decay is inferred immediately from the \cite[Theorem 1.1(b), pp.20--21]{trig3}.
\end{proof}

In the above semigroup approach, we have constructed solutions in the variable $p$ via the semigroup approach. We now describe how to pass from the variable $p$ to $\bu$ via the ODE in \eqref{main1**}$_1$. This observation is central to subsequent sections: When a given pressure function $p$ is obtained in a regularity class for which the ODE 
$$	\delta_1\bu_t+\bu=\cE^{-1}\mathbf F-\alpha \cE^{-1}(\nabla p)$$
can be readily interpreted, we obtain a mapping $\nabla p \mapsto \bu$. When one has decay estimates as above in Theorem \ref{decay} these can be pushed  from $p$ to $\bu$ in the solution to the full system \eqref{main1**}.

\begin{lemma}\label{ODEsolve} Consider the ODE in \eqref{StructureODE}. 
Letting $$Q =\delta^{-1}[-\alpha\cE^{-1}(\nabla p)+\cE^{-1}(\bF)],$$ we have
$\bu_t+\delta_1^{-1} \bu=Q,$ which can be solved as:
\begin{equation} \label{soluODE}
\bu(\mathbf x,t)=e^{-t/\delta_1}\bu(\mathbf x,0)+\int_0^te^{[\tau-t]/\delta_1}Q(\mathbf x,\tau)d\tau.\end{equation}
\end{lemma}
From Lemma \ref{ODEsolve}, one can pass the decay in Theorems \ref{dampedSG} and \ref{decay} (as well as Theorem \ref{Thm-DWave-Energy2}) from the variables $p$ to $\bu$ (and $\bu_t$). We provide an example for the state space norm below.

Take $\bF\equiv \mathbf 0$ and $S=0$. We will obtain a decay result for the displacement $\bu$ through the ODE.
 From (\ref{decayestimate}) we have that 
\begin{equation} \|y(t) \|_{Y} \leq M_0 e^{-\gamma t} \|y_0\|_Y \implies \|\grad p(t)\|_{L^2(\Omega)}  \leq M_0 e^{-\gamma t} \|y_0\|_Y \implies \|\cE^{-1}(\grad p(t))\|_{\mathcal{D}(\cE)} \leq M_0 e^{-\gamma t} \|y_0\|_Y.\end{equation}
Using (\ref{soluODE}) we obtain
$$\|\mathbf u(t) \|_{\mathcal{D}(\cE)}  \leq e^{-\frac{t}{\delta_1}}\|u_0\|_{\mathcal{D}(\cE)} + \frac{M_0}{1 - \gamma \delta_1}(e^{-t\gamma} - e^{-\frac{t}{\delta_1}})\|y_0\|_Y$$
Of course, the estimate above can be readily adjusted to accommodate the space in which $\bu_0$ is specified. 

\subsubsection{Incompressible Constituents}\label{czero}
\noindent We now take $c_0=\delta_2=0$ in \eqref{main1}, which, following the calculations of Section \ref{thissec}, yields the abstract dynamics:
\begin{equation} \label{ODEeq} \left[\delta_1A+{\alpha^2B}\right]p_t+Ap=\overline S,~~\text{ with }~\overline S = S-\alpha\nabla \cdot \cE^{-1}(\bF_t)+\delta_1S_t.\end{equation}
Note that $\overline S=\delta_1\widehat S$ from the previous section.
 In this case, we make a change of variable: \begin{equation}\label{CoV} q=[\alpha^2B+\delta_1A]p\end{equation} and proceed in the variable $q$. In this scenario, the operator $\alpha^2B+\delta_1A \in \mathscr L(V,V')$ and is boundedly invertible in this sense, following the properties of $A$ and $B$ in Section \ref{opsec}. Indeed, this follows immediately from: (i) Lax-Millgram on the strength of $A$ and (ii) the continuity and self-adjointness of $B$ on $L_0^2(\Omega)$. 
 
Under the change of variable, our abstract equation \eqref{ODEeq} can be written {as an ODE} in the variable $q$ as:
\begin{equation}\label{ode} q_t+A[(\alpha^2B+\delta_1A)^{-1}]q=\overline S \in L^2(0,T;V');~~q(0)=q_0\in V'.\end{equation}
The operator $R$ has that $$R\equiv A[(\alpha^2B+\delta_1A)^{-1}] \in \mathscr L(V') \cap \mathscr L(L_0^2(\Omega));$$ it is zeroth-order with $L^2(\Omega)$-adjoint
$$R^* = [(\alpha^2B+\delta_1A)^{-1}]A.$$ The ODE in \eqref{ode} can be interpreted either in the sense of $C(0,T;V')$ or $C(0,T;L_0^2(\Omega))$, depending on the regularity of the  data which compose $\overline S$---namely, whether we require $S, S_t$ and $\bF_t$ to take values in $L^2$ or $H^1$ type spaces.  In either case, the ODE can thus be solved in the context of a {\em uniformly continuous semigroup} \cite{pazy}. Here, the semigroup is $e^{-Rt} \in \mathscr L(X)$, where $X$ can be chosen to be $V'$ or $L^2_0(\Omega)$. 

For $\overline S \in L^2(0,T;V')$ and $q_0 \in V'$, the classical variation of parameters formula \cite{pazy} yields that $q \in H^1(0,T;V')$. However, $p \in H^1(0,T;V)$ is immediately obtained through inversion of the change of variables \eqref{CoV} $a.e.~t$. Then, the elasticity ODE can be explicitly solved in time as in Lemma \ref{ODEsolve},
providing $\bu \in H^1(0,T;\mathbf V)$. From this, we observe that $p(0)=p_0 \in V$ must be specified at the outset in order to possess an initial condition of the form $q(0)=q_0\in V'$. This reflects the fact that, since this cases reduces to an ODE, there is no spatial regularization provided by the pressure dynamics.
In this case, one only needs to specify $p(0)$ to obtain a solution to the abstract ODE in \eqref{ode}. This comes through the appearance of the combination $\delta_1\bu_t+\bu$ in the elasticity equation and the structure of the solution to the ODE in $\bu$. To recover the displacement variable $\bu$ from $p$, one must additionally specify $\bu(0)$ as well. Lastly, after solving the ODE in $q$ (and therefore for $p$), one can revert back to the ODE for $\bu$ to obtain additional temporal regularity (as a function of the regularity of $\bF$), since both sides of the equality below can be time-differentiated
 $$\delta_1\bu_t+\bu=\cE^{-1}[\bF-\alpha \nabla p].$$
 
Through this discussion, we have arrived at the following theorems.

\begin{theorem}[ODE Theorem] \label{ODEt1} Let $S \in H^1(0,T;V')$ and  $\bF \in H^1(0,T;\bV')$ ~(so that $\overline S = S-\alpha\nabla \cdot \cE^{-1}(\bF_t)+\delta_1S_t \in L^2(0,T;V'))$ and take $p_0 \in V$. 

Then there exists a unique (ODE) solution $p \in H^1(0,T; V)$ to \eqref{ode}, given by $$p(t) =[\alpha^2B+\delta_1A]^{-1}e^{-Rt}[\alpha^2B+\delta_1A]p_0.$$ 
\begin{itemize} 
\item If $S\in L^2(0,T;L^2_0(\Omega))$, then $p\in L^2(0,T; \mathcal D(A))$.
\item If $\bu_0 \in \mathcal D(\cE)$ and $\bF\in L^2(0,T;\mathbf L^2(\Omega))$, then there exists a unique weak solution $(p,\bu)$ to \eqref{main1}, with $p$ as before and $\bu \in H^1(0,T;\mathcal D(\cE))$. In this case, the formal energy equality in \eqref{EED1C0} holds for the weak solution $(p,\bu)$ with $c_0=0$.
\end{itemize}
\end{theorem}

\begin{theorem}[Regularity for ODE ]\label{ODEt2}
	Let $m\in\N$, $S\in H^1(0,T;H^{m}(\Omega)\cap L^2_0(\Omega))$ and ~$\bF\in H^1(0,T;\boldsymbol H^{m-1}(\Omega))$ (so that $\overline S = S-\alpha\nabla \cdot \cE^{-1}(\bF_t)+\delta_1S_t \in L^2(0,T;H^m(\Omega))),$ and take $p(0)\in H^{m+2}(\Omega)\cap V$. 
	
	Then the ODE solution $p$ to \eqref{ode} has $p\in H^1(0,T;H^{m+2}(\Omega))$.
\begin{itemize}
	\item If $S\in L^2(0,T;H^{m+1}(\Omega))$, then $p\in L^2(0,T; H^{m+3}(\Omega)))$.
	\item If  $\bu_0 \in H^{m+1}\cap\bV$, then the weak solution $(p,\bu)$ to \eqref{main1}, satisfies an additional regularity property $\bu \in H^1(0,T; \boldsymbol H^{m+1}(\Omega))$.
	\item If $\bu(0) \in H^{m+3}\cap\bV$ and $\bF\in L^2(0,T;\boldsymbol H^{m+1}(\Omega))$, then the weak solution $(p,\bu)$ to \eqref{main1}, satisfies an additional regularity property $\bu \in H^1(0,T; \boldsymbol H^{m+3}(\Omega))$.
\end{itemize}
\end{theorem}

We note that these theorems are obtained through the reduced formulation in \eqref{ODEeq}, which itself is obtained by time-differentiating the original equations (and the data). In some sense, then, we are lowering the regularity of the solution. However, some additional regularity is obtained a posteriori through elliptic regularity.

Finally, we discuss decay in the incompressible constituents case, which is not immediate. First, by Theorem \ref{Thm-DWave-Energy1}, the energy estimate \eqref{energyest} holds for $c_0=0$ solutions: 
\begin{align}\nonumber
||\bu||^2_{L^{\infty}(0,T;\bV)}+\delta_1||\bu_t||^2_{L^2(0,T;\bV)}+||p||^2_{L^2(0,T;V)} 
\lesssim  ||\bu_0||_V^2& +||S||^2_{L^2(0,T;V')} + ||\bF||^2_{L^2(0,T; \mathbf L^2(\Omega))}.
\end{align}
The above, of course, indicates dissipation in both variables $p$ and $\bu$, yet the pointwise-in-time quantity, $||p(t)||^2_{L_0^2(\Omega)}$, has disappeared. As this case considers an ODE for $q$ (on either $V'$ or $L^2_0(\Omega)$), the spectral properties of the operator $R$---on the respective space---would dictate decay in $q$. From that point of view, we only remark that: (i) the operator $R : V' \to V'$ is non-negative, and (ii) $0$ is not an eigenvalue of $R$.

However, we can  directly observe exponential decay in this case through a multiplier method for the whole system. Indeed, with solutions in hand from Theorems \ref{ODEt1} and \ref{ODEt2}, we can reconstruct the weak solution to \eqref{main1} and then selectively test the equations using wave-type stabilization arguments. We present this argument here.
\begin{theorem}[Exponential ODE Stability] \label{ODEdecay}
Consider $S \equiv 0$ and $\mathbf F \equiv 0$ in \eqref{main1} (so $\overline S \equiv 0$ in \eqref{ODEeq}). Suppose that $\bu_0, \bu_t(0) \in \mathbf V$ and $p_0 \in V$. Then, there exists $C>0$ and $\gamma>0$ so that we have the estimate for all $t \ge 0$:
\begin{equation}
||\bu(t)||_{\mathbf V}^2+||\mathbf u_t(t)||_{\mathbf V}^2+||p(t)||_{V}^2 \le C[||\bu_0||_{\mathbf V}^2+||\mathbf u_t(0)||_{\mathbf V}^2+||p_0||_{V}^2] e^{-\gamma t}
\end{equation}
\end{theorem}
\begin{proof}[Proof of Theorem \ref{ODEdecay}]
We will consider smooth solutions for formal calculations, which can then be extended by density in the standard way the final estimate on generalized (semigroup) solutions. The full system in strong form is: 
\begin{equation}\begin{cases}
\mathcal \cE\bu+\delta_1\cE\bu_t+\alpha \nabla p = 0\\ 
\alpha \nabla \cdot \bu_t+Ap=0.
\end{cases}
\end{equation}
First, we recall the standard energy estimate (as quoted above) which is obtained by testing the elasticity equation with $\bu_t$ and the pressure equation with $p$:
\begin{equation}\label{firstone}
\frac{1}{2}\dfrac{d}{dt}||\bu(t)||_{\bV}^2+\delta_1||\bu_t(t)||_{\bV}^2+k||p||^2_V = 0.
\end{equation}
Next, we differentiate the displacement equation in time, and test again with $\bu_t$, while testing the pressure equation with $p_t$ and adding:
\begin{equation}\label{secondone}
\dfrac{1}{2}||\bu_t(t)||_{\bV}^2+\dfrac{\delta_1}{2}\dfrac{d}{dt}||\bu_t(t)||_{\bV}^2+\dfrac{k}{2}\dfrac{d}{dt}||p||_V^2 = 0.
\end{equation}
Finally, we use the so called equipartition multiplier associated with the wave equation---namely, testing with $\bu$. (Note: This was in fact done in the energy estimates in \cite{bgsw}, but stability was not pursued there.) This yields the identity:
\begin{equation}
||\bu(t)||_{\bV}^2+\dfrac{\delta_1}{2}\dfrac{d}{dt}||\bu(t)||^2_{\bV}=-\alpha(\nabla p, \bu).
\end{equation}
From this we obtain the estimate by Young's inequality:
\begin{equation}
||\bu(t)||_{\bV}^2+\dfrac{\delta_1}{2}\dfrac{d}{dt}||\bu(t)||^2_{\bV} \le \dfrac{\alpha^2C_P}{2}||\nabla p(t)||^2+\dfrac{1}{2}||\bu(t)||^2_{\bV},
\end{equation}
where $C_P$ is a constant such that $||\bu||_{\mathbf L^2(\Omega)}^2 \le C_P ||\bu||_{\bV}^2$, which exists by virtue of Poincare and Korn. 
Now, absorbing the last term and then multiplying the resulting inequality by $k[\alpha^2C_P]^{-1}$, we obtain:
\begin{equation}\label{thirdone}
\frac{1}{2}\dfrac{k}{\alpha^2C_P}||\bu(t)||_{\bV}^2+\dfrac{\delta_1k}{2\alpha^2C_P}\dfrac{d}{dt}||\bu(t)||^2_{\bV} \le \dfrac{k}{2}||\nabla p(t)||^2,
\end{equation}
Adding \eqref{firstone}, \eqref{secondone}, and \eqref{thirdone}, and absorbing the final RHS term, we obtain the estimate:
\begin{equation}\small
\frac{1}{2}\dfrac{d}{dt}\left[\left(1+\dfrac{\delta_1k}{\alpha^2C_P}\right)||\bu(t)||_{\bV}^2+\delta_1||\bu_t(t)||_{\bV}^2+k||p||^2_V\right]+\left(\delta_1+\dfrac{1}{2}\right)||\bu_t||_{\bV}^2+\dfrac{k}{2}||p||^2_V+\dfrac{k}{2\alpha^2C_P}||\bu(t)||_{\bV}^2  \le 0
\end{equation}
Finally, if we define the quantity $$\mathbf E(t) \equiv \dfrac{1}{2}\left[ \left(1+\dfrac{\delta_1k}{\alpha^2C_P}\right)||\bu(t)||_{\bV}^2+\delta_1||\bu_t(t)||_{\bV}^2+k||p||^2_V\right],$$
then, we observe that there exists an $\gamma=\gamma(k,\delta_1,\alpha,C_P)$ with~ $\ds 0<\gamma < \min{\left\{1,\left(\dfrac{k}{\delta_1k+\alpha^2C_P}\right)  \right\}}$
so that we  have the Gr\"onwall type estimate:
$$\dfrac{d}{dt}\mathbf E(t) + \gamma\mathbf E(t) \le 0.$$
This implies the exponential decay:
$$\mathbf E(t) \le \mathbf E(0)e^{-\gamma t}.$$ From this, the final result of the theorem follows.
\end{proof}
\begin{remark} At the cost of scaling the RHS in the final estimate above, one can work with the more natural quantity 
$$\widetilde{\mathbf E}(t) \equiv \dfrac{1}{2}\left[||\bu(t)||_{\bV}^2+\delta_1||\bu_t(t)||^2_{\bV}+k||\nabla p||^2\right],$$
and obtain the analogous theorem.
\end{remark}

\subsection{Case 2: Adjusted Fluid Content, $\delta_2>0$}\label{thiss}
\noindent We now consider $\delta_1,~\delta_2>0$ in \eqref{main1}, which is to say, we allow for visco-elastic effects in the structural equation and we modify the definition of the fluid content of \eqref{main1}.
Thus, in this section, the fluid content will be given by \begin{equation} \label{contenthere} \zeta = c_0p+\alpha \nabla \cdot \bu +\delta_2\nabla \cdot \bu_t,\end{equation} where again we retain the coefficients $\delta_2,~\delta_1$ to observe their presence in the reduced dynamics. We  note that, for the dynamics to admit energy estimates, {\em we must observe the identity}
\begin{equation} \alpha \delta_1=\delta_2.\end{equation}
Alternatively, one obtains the above coefficient relation by formally mapping $\bu \mapsto \bu+\delta_1\bu_t$ in the derivation of the original Biot dynamics \eqref{bioteq}.

Note that {\em we will not make a distinction between $c_0\ge 0$ in this section}. Indeed, as we will see, upon making the abstract reduction, we will obtain precisely the same system (under different state variables) as the original Biot dynamics. The one distinction is that it will not be adequate to only specify $d_0$ or $p_0$ in the initial conditions. Indeed, we will require $\bu_0$ and one of $p_0$, $d_0$. Even considering the pressure equation alone, it will not be adequate to specify $d_0$ alone, as we will see.

Now, invoking the elasticity equation as in \eqref{StructureODE}, we can write (like before):
$$\delta_1\bu_t+\bu=\cE^{-1}\mathbf F-\alpha \cE^{-1}(\nabla p).$$
Taking the divergence of this equation, enforcing the condition $\alpha\delta_1=\delta_2$, and plugging into the fluid content expression in \eqref{contenthere}, we obtain a pressure equation from \eqref{main1}$_2$ of the form:
$$[c_0p+\alpha^2Bp+\nabla\cdot\mathcal E^{-1}\bF]_t+Ap=S,$$
which is rewritten as
\begin{equation}\label{reducedagain}[(c_0I+\alpha^2B)p]_t+Ap=\widetilde{S},\end{equation}
where $\widetilde{S}$ is as before in Section \ref{secbiot}.
We note that this pressure equation, with $\delta_1,\delta_2>0$, {has the exact same structure as the original implicit degenerate system} without visco-elasticity, as presented in \eqref{reducedBiot}. Thus, {the inclusion of viscous effects in both the fluid content and displacement equation} recovers the same implicit, degenerate (parabolic) dynamics given by Biot's poro-elastic dynamics. Although we have the same system abstractly, it is worth it to describe the resulting estimates and relevant quantities in this case. 

We note explicitly that, for the above dynamics in $p$, it is sufficient to specify either $p(0)=p_0$ or $\mathcal B p(0)=[(c_0I+\alpha^2B)p](0)$ in $L^2_0(\Omega)$---indeed, these are equivalent by the invertibility of $\mathcal B$ in that context. However, unlike the case of pure Biot dynamics ($\delta_1=\delta_2=0$), we cannot move from $d_0=\zeta(0) =  [c_0p+\alpha \nabla \cdot \bu +\delta_2\nabla \cdot \bu_t](0)$ to $p_0$ directly, as we must pass through the ODE for $\bu$ in \eqref{StructureODE}. Said differently: by moving to the abstract framework for the $\delta_1,~\delta_2>0$ dynamics, we can solve for $p$ in \eqref{reducedagain}. Then, solving the ODE in \eqref{StructureODE}---with a given initial condition $\bu_0$---we obtain the corresponding displacement solution $\bu$. As in previous cases, it is clear that given $p_0$ and $\bu_0$ the quantity $\bu_t(0)$ can be recovered through
$$\delta_1\bu_t(0)=-\bu_0-\alpha\cE^{-1} \nabla p_0+\cE^{-1}\bF(0) \in \bV.$$

This produces associated a priori estimates, albeit in an indirect way. One can immediately obtain energy estimates through the multiplier method, as in previous sections. However, to obtain a priori estimates (e.g., on approximants), one will test the pressure equation with $p$ in \eqref{main1}$_2$ and the displacement equation \eqref{main1}$_1$ with  $\delta_1\bu_{tt}+\bu_t$. 
In this step, {one again observes the necessary requirement} that $\alpha\delta_1=\delta_2$. The resulting identities are:
\begin{align*}
\dfrac{1}{2}\dfrac{d}{dt}e\big(\delta_1\bu_t+\bu_t,\delta_1\bu_t+\bu\big)-\alpha (p,\delta_1\nabla \cdot \bu_{tt}+\nabla \cdot \bu_t)=&~\big\langle \bF(\tau),[\delta_1\bu_t(\tau)+\bu(\tau)]\big\rangle_{\bV'\times \bV}\Big|_{\tau = 0}^{\tau = t}- \big\langle \bF_t, [\delta_1\bu_t+\bu]\big\rangle_{\bV'\times \bV} \\
 (\nabla \cdot[\delta_2\bu_{tt}+\alpha \bu_t],p)+\dfrac{c_0}{2}\dfrac{d}{dt}|| p||^2+a(p,p)=&~\langle S,p\rangle_{V' \times V}.
 \end{align*}
 These identities, along with the application of the abstract theorem in Section \ref{secbiot}, yield the central theorem for this case. We refer to the Appendix for the definition of weak solutions in the case of the reduced, implicit formulation in \eqref{reducedagain}.
\begin{theorem}[Visco-elastic Solutions with Modified Fluid Content] \label{finth} Suppose that $S \in L^2(0,T;V')$ and $\bF \in H^1(0,T;\bV')$. Take $p_0 \in L_0^2(\Omega)$. Suppose $\delta_2>0$ with $c_0\ge 0$ and enforce the condition that $\delta_2=\alpha \delta_1$. Then there exists a unique weak solution $p\in L^2(0,T;V)$ satisfying the reduced, implicit formulation \eqref{reducedagain}. 

If $\bu_0 \in \bV$, then (with $p$ the same as above) there exists a unique weak solution $(\bu,p)\in C([0,T];\bV)\times L^2(0,T;V)$ to \eqref{main1} satisfying the energy inequality {\small
\begin{equation}\label{finest} ||\delta_1\bu_t+\bu||^2_{L^{\infty}(0,T;\bV)}+c_0||p||^2_{L^{\infty}(0,T;L_0^2(\Omega))} +||p||^2_{L^2(0,T;V)} \lesssim ||\bu_0||_{\bV}^2+c_0||p_0||^2_{L_0^2(\Omega)}+||S||^2_{L^2(0,T;V')}+C(T)||\bF||^2_{H^1(0,T;\bV')}.\end{equation}}
\end{theorem}

In this above framework---as we have reduced to the same abstract theory for classical Biot with $\delta_1=\delta_2=0$---we can accordingly discuss parabolic estimates and smooth solutions. We do not repeat the statements here, but refer back to Theorem \ref{th:main2} and Theorem \ref{strongBiot}, which can be analogously adopted here. 

\begin{remark} If one allows for the possibility that the coefficient $\delta_2$ is fully independent, renaming $\delta_2$ as $\tilde \alpha$, one obtains the system:
\begin{equation}\label{main2} \begin{cases}
\mathcal \cE\bu+\delta_1\cE\bu_t+\alpha \nabla p = \mathbf F \\ 
[c_0p+\alpha \nabla \cdot \bu+\tilde\alpha\nabla \cdot \bu_{t}]_t+Ap=S.
\end{cases}
\end{equation}
Accordingly, we can obtain a reduced equation which is not closed in $p$:
\begin{equation}\nonumber
[c_0p+\dfrac{\alpha\tilde\alpha}{\delta_1}B p]_t+\left(\alpha-\dfrac{\tilde\alpha}{\delta_1}\right)\dfrac{\alpha}{\delta_1} B p- \left(\alpha-\dfrac{\tilde\alpha}{\delta_1}\right)\delta_1^{-1}\nabla\cdot \bu + Ap = S-\dfrac{\tilde \alpha}{\delta_1}\nabla \cdot \cE^{-1}\bF_t-\delta_1^{-1}\left(\alpha-\dfrac{\tilde\alpha}{\delta_1}\right)\nabla\cdot \cE^{-1}\bF
\end{equation}
The ODE for $\bu$ can be solved as before, which produces the equation:
 {\small
\begin{align}
\left[\left(c_0I+\dfrac{\alpha \tilde \alpha}{\delta_1}B\right)p\right]_t+\left[A+\left(\dfrac{\alpha^2}{\delta_1}-\dfrac{\alpha}{\tilde \alpha}\right)B\right]p - \left(\dfrac{\alpha^2}{\delta_1^2}-\dfrac{\alpha\tilde\alpha}{\delta_1^3}\right)\int_0^te^{-(t-\tau)/\delta_1}Bp(\tau) d\tau &\\[.2cm] \nonumber
=  ~S-\dfrac{\tilde\alpha}{\delta_1}\nabla \cdot \cE^{-1}\bF+\left(\dfrac{\tilde\alpha}{\delta_1^2}-\dfrac{\alpha}{\delta_1}\right)e^{-t/\delta_1}\nabla \cdot \bu_0  +\left(\dfrac{\tilde\alpha}{\delta_1^2}-\dfrac{\alpha}{\delta_1}\right)\nabla \cdot \cE^{-1}\bF-\left(\dfrac{\tilde\alpha}{\delta_1^3}-\dfrac{\alpha}{\delta_1^2}\right)&\int_0^te^{-(t-\tau)/\delta_1}\nabla\cdot \cE^{-1}\bF(\tau)d\tau
\end{align}}
This is also a visco-elastic equation which can be solved by the methods in this paper, but we do not pursue this here. 
In this case we note the emergence of additional terms which vanish when we enforce the condition $\tilde \alpha = \delta_1 \alpha$. \end{remark}

\section{Remarks on Secondary Consolidation}\label{secondary}\noindent
In some sense, the effects of {\em secondary consolidation} of soils (i.e., {\em creep}) can be thought of as partial visco-elasticity \cite{applied}. Therefore, for completeness, we add some remarks here on the nature of associated solutions and estimates. In the case of secondary consolidation, as it is described in \cite{show1}, we ``regularize" only the divergence term in the momentum equation (omitting $\Delta\bu_t$ from the full visco-elastic terms). 
Thus, we consider the Biot system with secondary consolidation effects, and, as before, we allow both definitions of the fluid content based strictly on mathematical grounds. 
So we have: \begin{equation} \begin{cases}
-\lambda^*\nabla \nabla \cdot \bu_t-\mu\Delta \bu-(\lambda+\mu)\nabla \nabla \cdot \bu+\alpha \nabla p = \mathbf F \\ 
[c_0p+\alpha \nabla \cdot \bu+\delta_2\nabla \cdot \bu_{t}]_t+Ap=S 
\end{cases}
\end{equation}
In these short sections, we will track the impact of this ``partial viscoelastic" $\lambda^*$ term. We note that there is no need to consider the case with  full visco-elasticity $\delta_1>0$ and secondary consolidation, as the latter would be redundant.

\subsection{Traditional Fluid Content: $\delta_2=0$}
\noindent This is the secondary consolidation model as explicitly discussed in \cite{show1,applied}. 
\begin{equation} \begin{cases}
-\lambda^*\nabla \nabla \cdot \bu_t-\mu\Delta \bu-(\lambda+\mu)\nabla \nabla \cdot \bu+\alpha \nabla p = \mathbf F \\ 
[c_0p+\alpha \nabla \cdot \bu]_t+Ap=S 
\end{cases}
\end{equation}
The well-posedness of weak solutions is given in \cite{show1}. We mention here that, using the standard multipliers for weak solutions (justified by the approach in \cite{bmw}) one obtains the 
  following estimate on solutions:
\begin{equation}\small
||\bu||_{L^{\infty}(0,T;\mathbf V)}^2+c_0||p||_{L^{\infty}(0,T;L^2(\Omega))}^2 + \lambda^*||\nabla \cdot \bu_t||_{L^2(0,T;L^2(\Omega))}^2+||p||^2_{L^2(0,T;V)} \lesssim ||\bu_0||^2_{\bV}+ ||S||_{L^2(0,T;V')}^2+||\bF||_{H^1(0,T;\mathbf V')}^2
\end{equation}
A partial ``visco-elastic" effect of secondary consolidation is immediate obviated:  the additional damping/dissipation term above for $\nabla \cdot \bu_t$. Upon temporal integration, we will obtain the additional property of weak solutions that $\nabla \cdot \bu_t \in L^2(0,T;L^2(\Omega))$. This term represents a certain ``smoothing" as well, as $\nabla\cdot \bu_t$ has been boosted from $L^2(0,T;V') \mapsto L^2(0,T;L^2(\Omega))$ by the presence of $\lambda^*>0$. 
We note that since the fluid content $c_0p+\nabla \cdot \bu$ lies in $H^1(0,T;V')$ (via the pressure equation), we can now extract $c_0p \in L^2(0,T;V')$, which is not obvious when $\lambda^*=0$, since we cannot decouple the two terms in the sum for the fluid content. 

\subsubsection{Incompressible Constituents}
\noindent In the case of $c_0=0$,  we observe some partial regularization of the dynamics for $\lambda^*>0$; this is explicitly mentioned in \cite{show1}, and we expand upon it here.

 Note that from the pressure equation, we can write:
$$Ap=S-\alpha \nabla \cdot \bu_t,$$ from which elliptic regularity can be applied---for weak solutions---when $\Omega$ is sufficiently regular and $S \in L^2(0,T;L^2(\Omega))$. Then, with $\nabla \cdot \bu_t \in L^2(0,T;L^2(\Omega))$ as described above, we observe a ``boost" $p \in L^2(0,T;V) \mapsto p \in L^2(0,T;\mathcal D(A))$ through elliptic regularity applied a.e. $t$. But this cannot be pushed on the momentum equation, owing to the addition of the secondary consolidation term:
$$\cE(\bu)=\bF-\alpha \nabla p+\lambda^*\nabla\nabla\cdot \bu_t.$$
This is to say that the regularity gain in $p$ is not realized for the displacement $\bu$ through the momentum equation.

One further observation, in this case, is a particular representation of the system which is not available in other cases.  Noting that $p=A^{-1}[S-\alpha \nabla \cdot \bu_t]$, one can plug this into the elasticity equation to obtain:
$$-\big[\alpha^2\nabla A^{-1}\text{div}+\lambda^* \nabla \text{div}\big]\bu_t+\cE(\bu)=\bF-\alpha \nabla A^{-1}S.$$
 This is an implicit equation {\em directly in} $\bu$ which can be analyzed in the framework of implicit, degenerate equations \cite{show1,auriault,indiana}; we do not pursue this line of investigation here.

\subsubsection{Compressible Constitutents} 
In the case of  $c_0>0$, \cite{show1} observes that the effect  secondary consolidation is {\em de-regularizing}. This is observed in hindering the discussion in the previous section, namely the pressure equation now reads as:
$$Ap=S-\alpha \nabla \cdot \bu_t-c_0p_t.$$
The effect of secondary consolidation through $\lambda^*$ (the boosting of $\nabla\cdot \bu$ to $L^2(0,T;L^2(\Omega))$ is lost, since we can only conclude that $c_0p_t \in L^2(0,T;V')$ rather than $L^2(0,T;L^2(\Omega))$. 
Thus there is neither smoothing in $p$ nor $\bu$ in this case.

\subsection{Adjusted Fluid Content}
Finally, we observe that in the case of {\em adjusted fluid} content, we obtain the natural analog to our earlier discussions. 
Taking $\delta_2>0$, we consider the system:
\begin{equation} \begin{cases}
-\lambda^*\nabla \nabla \cdot \bu_t+\cE\bu+\alpha \nabla p = \mathbf F \\ 
[c_0p+\alpha\nabla \cdot \bu+\delta_2\nabla \cdot \bu_t]_t+Ap=S 
\end{cases}
\end{equation}
As above, we invoke (as before) the test function $\delta_2\bu_{tt}+ \bu_t$ in the elasticity equation, and $p$ in the pressure equation. This provides an identical estimate as that in Theorem \ref{finth} with the additional property that $\nabla \cdot \bu_t \in L^2(0,T;L^2(\Omega))$, and the associated term 
~$\ds \lambda^*||\nabla \cdot \bu_t||_{L^2(0,T;L^2(\Omega))}^2$
appears on the LHS of \eqref{finest} in Theorem \ref{finth}. Again, then, we see the effect of secondary consolidation as that of partial damping.

\section{Summary and Conclusions}
In this note, we characterized linear poro-visco-elastic systems across several parameter regimes:
\begin{equation}\label{main3} \begin{cases}
\mathcal \cE\bu+\delta_1\cE\bu_t+\alpha \nabla p = \mathbf F \\ 
[c_0p+\alpha \nabla \cdot \bu+\delta_2\nabla \cdot \bu_{t}]_t+Ap=S.
\end{cases}
\end{equation}
We began with the traditional Biot system ($\delta_1=\delta_2=0$), i.e., no Kelvin-Voigt visco-elastic effects, and recapitulated existence  results and estimates for weak solutions, as well as solutions with higher regularity in Section \ref{secbiot}. Using this as a jumping-off point, we considered the addition of linear Kelvin-Voigt type (strong) dissipation in the Lam\'e system \eqref{main3}. {Our central focuses were in the well-posedness and regularity of solutions across all parameter regimes, as well as the clear determination of the abstract structure of the problem, including the discernment of the appropriate initial quantities}. Our approach included {providing clear a priori estimates on solutions, where they were illuminating. We employed an operator-theoretic framework} inspired by \cite{auriault,show1} and developed in \cite{bmw} which we introduced in Section \ref{opsec}. The central operators were $A$ (a Neumann Laplacian) and $B$ (a zeroth order, nonlocal pressure-to-divergence operator). 

We then considered the model with $\delta_1>0$ and $\delta_2=0$, which is to say we left the fluid-content unaltered in our addition of strong damping. {We first gave a well-posedness result which was valid for both compressible constituents $c_0>0$ and incompressible constituents $c_0=0$} in Section \ref{thissec}. We then distinguished between these two cases. {We determined that for $c_0>0$, the system constitutes a strongly damped hyperbolic-type system.} It was important, in this case, to distinguish results based on which initial quantities were specified. Regardless, the regularity of solutions was made clear, and the parabolicity of the system was detailed in several ways. In the case {when $c_0=0$, we observed that the abstract, reduced version of the dynamics constituted an ODE in a Hilbert space} of our choosing (either $L^2_0(\Omega)$ or an $H^{-1}(\Omega)$ type space, $V'$). We exploited the ODE nature of the dynamics to produce a clear well-posedness and regularity result.

In the case when $\delta_1,\delta_2>0$ (i.e., the adjusted fluid content), we observed that {the abstract reduction of the system brings the dynamics back to the traditional Biot-structure.} In other words, by adding visco-elasticity to the displacement equation ($\delta_1>0$), as well as adjusting the fluid content ($\delta_2>0$), we do not observe additional effects from the visco-elastic damping---rather, {we obtain the same qualitative results for the solution as we had for Biot's original dynamics.} Noting a small difference in which initial quantities must be prescribed, we presented a well-posedness theorem, with relevant a priori estimates. 

Finally, in Section \ref{secondary}, we provide some small remarks on {\em partial visco-elasticity}, known in soil mechanics as {\em secondary consolidation}. The main focus of this section was to provide clear a priori estimates on solutions, indicating precisely how dissipation is introduced into the system through secondary consolidation effects. Additionally, we corroborate  remarks in \cite{show1} concerning the extent to which this sort of partial visco-elasticity can be, in fact, partially regularizing (when $c_0=0$) and de-regularizing ($c_0>0$).

The Appendix serves to provide a small overview of the standard theory of weak solutions for implicit, degenerate evolution equations \cite{bmw, auriault,indiana}, and is taken from \cite{showmono}.

We believe that the work presented here, as it is in the spirit of \cite{auriault,indiana,show1}, will be of interest to researchers working on applied problems in poro-elasticity. In particular, as the effects of visco-elasticity are prominent in biological sciences, those who work on biologically-motivated Biot models may find the results presented herein useful. Indeed, to the best of our knowledge, we have provided the first elucidation of the mathematical effects of linear visco-elasticity, when included in linear poro-elastic dynamics. 

\section{Appendix: Abstract Framework for Weak Solutions}
\noindent Let $V$ be a separable Hilbert space with dual $V'$ (not identified with $V$ ). Assume $V$ densely and continuously includes into another Hilbert space $H$, which {\em is} identified with its dual: ~$V \hookrightarrow H \equiv H' \hookrightarrow V'.$
            We denote the inner-product in $H$ simply as $(\cdot,\cdot)$, with $(h,h)=||h||_H^2$ for each $h \in H$. Similarly, we denote the $V'\times V$ duality pairing as ~$\langle \cdot, \cdot \rangle$. (For $h\in H$, we identify $\langle h, h \rangle = ||h||_H^2$ as well.)
            Assume that $\mathcal A \in \mathscr L(V,V')$  and $\cB \in \mathscr L(H)$. 
            Finally, suppose that $d_0 \in V'$ and $S \in L^2(0,T;V')$ are the specified data. 
            
            In this setup, we can define the weak (implicit-degenerate) Cauchy problem to be solved as:
            
    \noindent         Find $w \in L^2(0,T;V)$ such that
            \begin{equation}\label{ImplicitCauchy}
            \begin{cases}
            \dfrac{d}{dt}[\cB w]+\cA w=S \in ~L^2(0,T;V') \\[.2cm] \ds
         \lim_{t\searrow 0}[\cB w(t)]=~d_0 \in V'.
            \end{cases}
            \end{equation}
            The time derivative above is taken in the sense of $\mathscr D'(0,T)$, and since such a solution would have $\cB u \in H^1(0,T;V')$ (with the natural inclusion $V \hookrightarrow V'$ holding), $\cB u$ has point-wise (in time) values into $V'$ and the initial conditions makes sense through the boundedness of $\cB$ with $H\hookrightarrow V'$.             

The following generation theorem is adapted from~\cite[III.3, p.114--116]{showmono} for weak solutions, and produces weak solvability of \eqref{ImplicitCauchy} in a straight-forward way:
\begin{theorem}\label{ImplicitGeneration}
    Let $\cA,\cB$ be as above, and assume additionally that they are self-adjoint and monotone (in the respective sense, $\cA:V\to V'$ and $\cB:H\to H$). Assume further that there exists $\lambda,c >0$ so that $$2\langle \mathcal Av,v\rangle +\lambda(\mathcal Bv,v)\ge c||v||_V^2,~~~\forall~~v \in V.$$
    
    Then, given $\mathcal B w(0) = d_0 \in H$ and $ S \in L^2(0,T;V')$, there exists a unique weak solution to \eqref{ImplicitCauchy} satisfying
    \begin{align}
    ||w||^2_{L^2(0,T;V)} \le&~ C(\lambda,c)\left[||S||^2_{L^2(0,T;V')}+(d_0,w(0))_H\right] .
    \end{align}
\end{theorem}
The assumption in this theorem is that there exists a $w(0) \in V$ so that $\cB w(0) = d_0$, the given initial data. In more recent work applying this theorem \cite{bmw}, we need not assume the existence of such $w(0)$. See Theorem \ref{th:main2}.

One may also consult the implicit semigroup theory presented in \cite[Section 5]{show1} and \cite[IV.6]{showmono}, in particular for a discussion of smoother solutions.


{\small
\begin{thebibliography}{99}


\bibitem{auriault} Auriault, J.L., 1980. Dynamic behaviour of a porous medium saturated by a Newtonian fluid. {\em International Journal of Engineering Science}, 18(6), pp.775--785.

\bibitem{trig3} Balakrishnan, A.V. and Triggiani, R., 1993. Lack of generation of strongly continuous semigroups by the damped wave operator on $H \times H$ (or: The little engine that couldn't). {\em Applied Mathematics Letters}, 6(6), pp.33--37.

\bibitem{banks2}
\newblock Banks, H.T.,  Bekele-Maxwell, K.,  Bociu, L., Noorman, M., and Guidoboni, G., 2017.
\newblock {Sensitivity Analysis
in Poro-Elastic and Poro-Visco-Elastic Models with Respect to Boundary Data}, 
\newblock {\em Quarterly of
Applied Mathematics} 75, pp.697--735.


\bibitem{Biot}
\newblock Biot, M.A., 1941.
\newblock General theory of three-dimensional consolidation, 
\newblock {\em J. Appl. Phys.}, 12(2), pp. 155--164.

\bibitem{Biotporovisco}
\newblock Biot, M.A., 1956. 
\newblock Theory of deformation of a porous viscoelastic anisotropic solid. 
\newblock {\em J. of Applied physics}, 27(5), pp.459--467.


\bibitem{both2} Both, J.W., Borregales, M., Nordbotten, J.M., Kumar, K. and Radu, F.A., 2017. Robust fixed stress splitting for Biot's equations in heterogeneous media. Applied Mathematics Letters, 68, pp.101--108.

\bibitem{both3} J.W. Both, K. Kumar, J.M. Nordbotten, F.A. Radu, The gradient flow structures of thermo-poro-visco-elastic processes
in porous media, 2019, arXiv preprint arXiv:1907.03134

\bibitem{both4} Both, J.W., Pop, I.S. and Yotov, I., 2021. Global existence of weak solutions to unsaturated poroelasticity. ESAIM: Mathematical Modelling and Numerical Analysis, 55(6), pp.2849-2897.

\bibitem{bcmw} Bociu, L., \v{C}ani\'c, S., Muha, B. and Webster, J.T., 2021. Multilayered poroelasticity interacting with Stokes flow. {\em SIAM Journal on Mathematical Analysis}, 53(6), pp.6243-6279.

\bibitem{bgsw}
\newblock  Bociu, L., Guidoboni, G.,  Sacco, R., and Webster, J.T., 2016.
\newblock Analysis of Nonlinear Poro-Elastic and Poro-Visco-Elastic Models,
\newblock {\em ARMA}, 222, pp. 1445--1519

\bibitem{MBE2}
\newblock Bociu, L., Guidoboni, G.,  Sacco, R., and Verri, M., 2019.
\newblock {\em{On the role of compressibility in poroviscoelastic models}}, 
\newblock Mathematical Biosciences and Engineering, 16(5), 6167--6208.

\bibitem{bmw} Bociu, L., Muha, B. and Webster, J.T., 2022. Weak solutions in nonlinear poroelasticity with incompressible constituents. {\em Nonlinear Analysis: Real World Applications}, 67, p.103563.

\bibitem{bociuno} 
\newblock Bociu, L. and Noorman, M., 2019.
\newblock { Poro-Visco-Elastic Models in Biomechanics: Sensitivity Analysis},
\newblock {\em Communications in Applied Analysis}, 23(1), pp.61--77.

\bibitem{bsAWM}
\newblock Bociu, L. and Strikwerda,  S., 2022.
\newblock Poro-Visco-Elasticity in Biomechanics: Optimal Control. 
\newblock In: Research in Mathematics of Materials Science (pp. 103--132). Cham: Springer International Publishing.


\bibitem{bw} Bociu, L. and Webster, J.T., 2021. Nonlinear quasi-static poroelasticity. {J. of Differential Equations}, 296, pp.242--278.

\bibitem{cao2}
\newblock Cao, Y., Chen, S., and Meir, A.J., 2014.
\newblock Steady flow in a deformable porous medium,
\newblock {\em Math. Meth. Appl. Sci.}, 37, pp.1029--1041.

\bibitem{applied2} Castelletto, N., Klevtsov, S., Hajibeygi, H. and Tchelepi, H.A., 2019. Multiscale two-stage solver for Biot's poroelasticity equations in subsurface media. {\em Computational Geosciences}, 23, pp.207--224.

\bibitem{trig2} Chen, S.P. and Triggiani, R., 1989. Proof of extensions of two conjectures on structural damping for elastic systems. {\em Pacific Journal of Mathematics}, 136(1), pp.15--55.

\bibitem{chenruss} Chen, G. and Russell, D.L., 1982. A mathematical model for linear elastic systems with structural damping. {\em Quarterly of Applied Mathematics}, 39(4), pp.433--454.


\bibitem{coussy}
\newblock Coussy, O., 2004. 
\newblock {\em Poromechanics}.
\newblock John Wiley \& Sons.


\bibitem{detournay}
\newblock Detournay, E. and Cheng, A.H.-D., 1993.
\newblock Fundamentals of poroelasticity, 
\newblock Chapter 5 in {\em Comprehensive
Rock Engineering: Principles, Practice and Projects, Vol. II, Analysis and Design
Method}, ed. C. Fairhurst, Pergamon Press, 113--171.

\bibitem{EveVernescu91}
\newblock Ene, H. I. and Vernescu, B., 1992.,
\newblock Viscosity dependent behaviour of viscoelastic porous media.
\newblock Chapter 3 in {\em Asymptotic Theories for Plates and Shells},  Pitman Research Notes in Mathematics Series, 319.,  Chapman and Hall/CRC, 1995.

\bibitem{numerical2} Fred, V., Rodrigo, C., Gaspar, F. and Kundan, K., 2021. Guest editorial to the special issue: computational mathematics aspects of flow and mechanics of porous media. {\em Computational Geosciences}, 25(2), pp.601--602.

\bibitem{gaspar} Gaspar, F.J., Gracia, J.L., Lisbona, F.J. and Vabishchevich, P.N., 2008. A stabilized method for a secondary consolidation Biot's model. {\em Numerical Methods for Partial Differential Equations: An International Journal}, 24(1), pp.60--78.

\bibitem{GilbertAcoustic}
\newblock Gilbert, R.P. and Panchenko A., 2004.,
\newblock Effective acoustic equations for a two-phase medium with microstructure.
\newblock {\em Mathematical and Computer Modelling} 39(13), pp. 1431--1448

\bibitem{applied3} Haghighat, E., Amini, D. and Juanes, R., 2022. Physics-informed neural network simulation of multiphase poroelasticity using stress-split sequential training. {\em Computer Methods in Applied Mechanics and Engineering}, 397, p.115141.


\bibitem{oldthermo}  Henry DB, Perissinitto A Jr. Lopes O., 1988. On the essential spectrum of a semigroup of thermoelasticity. {\em Nonlinear
Analysis: Theory Methods Appl.} 21, 65--75.

\bibitem{showrecent} Hosseinkhan, A. and Showalter, R.E., 2021. Biot-pressure system with unilateral displacement constraints. Journal of Mathematical Analysis and Applications, 497(1), p.124882.

\bibitem{IKLT14}
\newblock Ignatova, M., Kukavica, I., Lasiecka I. and Tuffaha A., 2014.
\newblock On well-posedness and small data global existence for an interface damped free boundary fluidâstructure model.
\newblock {\em Nonlinearity} 27(3), p. 467 

\bibitem{redbook} Lasiecka, I. and Triggiani, R., 2000. Control theory for partial differential equations (Vol. 1). Cambridge: Cambridge University Press.

\bibitem{biobiot} Lee, J.J., Piersanti, E., Mardal, K.A. and Rognes, M.E., 2019. A mixed finite element method for nearly incompressible multiple-network poroelasticity. {\em SIAM J. on Scientific Computing}, 41(2), pp.A722--A747.

\bibitem{MT21}
\newblock Maity, D. and Takahashi, T., 2021.
\newblock $L^p$ theory for the interaction between the incompressible Navier-Stokes system and a damped plate.
\newblock {\em J. Math. Fluid Mech.} 23(4), Paper No. 103

\bibitem{MRT} Mardal, K.A., Rognes, M.E. and Thompson, T.B., 2021. Accurate discretization of poroelasticity without Darcy stability: Stokes-Biot stability revisited. BIT Numerical Mathematics, 61, pp.941-976.

\bibitem{MeiVernescu10}
\newblock Mei, C.C. and Vernescu, B., 2010.,
\newblock Homogenization methods for multiscale mechanics.
\newblock World Scientific Publishing Co. Pte. Ltd., Hackensack, NJ, 2010. xviii+330 pp

\bibitem{Meirm}
\newblock Meirmanov, A.M. and Zimin, R., 2012. 
\newblock Mathematical models of a diffusion-convection in porous media. 
\newblock {\em Electronic Journal of Differential Equations}, 2012(105), pp.1--16.


\bibitem{mow}
\newblock Mow, V.C., Kuei, S.C., Lai,  W.M., and Armstrong, C.G., 1980.
\newblock Biphasic creep and stress relaxation
of articular cartilage in compression: Theory and experiments, 
\newblock {\em ASME J. Biomech. Eng.}, 102, pp.73--84.

\bibitem{BorSun}
\newblock Muha, B., \v Cani\' c, S., 2013.
\newblock Existence of a weak solution to a nonlinear fluid-structure interaction problem modeling the flow of an incompressible, viscous fluid in a cylinder with deformable walls.
\newblock {\em Archive for Rational Mechanics and Analysis}, 207(3), pp. 919--968

\bibitem{applied}
Murad, M.A. and Cushman, J.H., 1996. Multiscale flow and deformation in hydrophilic swelling porous media. {\em International Journal of Engineering Science}, 34(3), pp.313-338.

\bibitem{nia}
\newblock Nia, H.T., Han, Ll,  Li, Y., Ortiz, C., and Grodzinsky, A., 2011.
\newblock Poroelasticity of cartilage at the nanoscale, 
\newblock {\em Biophys. J.}, 101, pp.2304--2313.



\bibitem{ozkaia}
\newblock Ozkaya, N., Nordin, M., Goldsheyder, D., and Leger, D., 1999.
\newblock {\em Fundamentals of Biomechanics. Equilibrium, Motion, and Deformation, }
\newblock Springer, New York.


\bibitem{pazy} Pazy, A., 2012. {\em Semigroups of linear operators and applications to partial differential equations} (Vol. 44). Springer Science \& Business Media.

\bibitem{numerical1} Rodrigo, C., Gaspar, F.J., Hu, X. and Zikatanov, L.T., 2016. Stability and monotonicity for some discretizations of the Biot's consolidation model. {\em Computer Methods in Applied Mechanics and Engineering}, 298, pp.183--204.

\bibitem{rohan} Rohan, E., Shaw, S., Wheeler, M.F. and Whiteman, J.R., 2013. Mixed and Galerkin finite element approximation of flow in a linear viscoelastic porous medium. {\em Computer Methods in Applied Mechanics and Engineering}, 260, pp.78-91.

\bibitem{wheel} Russell, T.F. and Wheeler, M.F., 1983. Finite element and finite difference methods for continuous flows in porous media. In: {\em The mathematics of reservoir simulation} (pp. 35--106). Society for Industrial and Applied Mathematics.

\bibitem{GGbook}
\newblock Sacco,  R., Guidoboni, G., and Mauri, A.G., 2019.
\newblock {\em A Comprehensive Physically Based Approach to Modeling in Bioengineering and Life Sciences,}
\newblock Elsevier Academic Press.

\bibitem{Sanchzez-Palencia}
\newblock Sanchez-Palencia, E., 1980. 
\newblock Non-homogeneous media and vibration theory. 
\newblock {\em Lecture Notes in Physics} 127, Springer-Verlag

\bibitem{indiana}
Showalter, R.E., 1974. Degenerate evolution equations and applications. {\em Indiana University Mathematics J.}, 23(8), pp.655-677.

\bibitem{showmono}
\newblock Showalter, R.E., 1996.
\newblock {\em Monotone Operators in Banach Space and Nonlinear Partial Differential Equations,}
\newblock AMS, Mathematical Surveys and Monographs, 49. 

\bibitem{show1} 
\newblock Showalter, R.E., 2000.
\newblock Diffusion in poro-elastic media,
\newblock {\em J. Mathematical Analysis and Application}, 251, pp. 310--340.


\bibitem{Terzaghi}
\newblock Terzaghi, K., 1925.
\newblock {\em Principle of Soil Mechanics,}
\newblock Eng. News Record, A Series of Articles.

\bibitem{mikelic} van Duijn, C.J. and Mikelic, A., 2019. Mathematical Theory of Nonlinear Single-Phase Poroelasticity. Preprint hal-02144933, Lyon June.

\bibitem{MBE1}
\newblock Verri, M., Guidoboni, G.,  Bociu, L.,  and Sacco, R., 2018.
\newblock {\em The Role of Structural Viscoelasticity in Deformable Porous Media with Incompressible Constituents: Applications in Biomechanics}, 
\newblock Mathematical Biosciences and Engineering, Volume 15, Number 4, 933--959.


\bibitem{Vist}
\newblock Visintin, A., 2012. 
\newblock On the homogenization of visco-elastic processes. 
\newblock {\em The IMA Journal of Applied Mathematics}, 77(6), pp.869--886.

\bibitem{zenisek}
\newblock Zenisek, A., 1984.
\newblock The existence and uniqueness theorem in Biot's consolidation theory, 
\newblock {\em Appl. Math.}, 29,, pp.194--211.

\end{thebibliography}}
\end{document}